\documentclass[oneside,british,english]{amsart}
\usepackage[T1]{fontenc}
\usepackage[latin9]{inputenc}
\usepackage{refstyle}
\usepackage{amsthm}
\usepackage{amssymb}
\usepackage{esint}

\makeatletter

\AtBeginDocument{\providecommand\propref[1]{\ref{prop:#1}}}
\AtBeginDocument{\providecommand\lemref[1]{\ref{lem:#1}}}
\AtBeginDocument{\providecommand\secref[1]{\ref{sec:#1}}}
\AtBeginDocument{\providecommand\thmref[1]{\ref{thm:#1}}}
\AtBeginDocument{\providecommand\subref[1]{\ref{sub:#1}}}
\RS@ifundefined{subref}
  {\def\RSsubtxt{section~}\newref{sub}{name = \RSsubtxt}}
  {}
\RS@ifundefined{thmref}
  {\def\RSthmtxt{theorem~}\newref{thm}{name = \RSthmtxt}}
  {}
\RS@ifundefined{lemref}
  {\def\RSlemtxt{lemma~}\newref{lem}{name = \RSlemtxt}}
  {}

\numberwithin{equation}{section}
\numberwithin{figure}{section}
\theoremstyle{plain}
\newtheorem{thm}{\protect\theoremname}[section]
  \theoremstyle{remark}
  \newtheorem{rem}[thm]{\protect\remarkname}
  \theoremstyle{plain}
  \newtheorem*{lem*}{\protect\lemmaname}
  \theoremstyle{plain}
  \newtheorem{prop}[thm]{\protect\propositionname}
  \theoremstyle{definition}
  \newtheorem{example}[thm]{\protect\examplename}
  \theoremstyle{definition}
  \newtheorem{defn}[thm]{\protect\definitionname}
  \theoremstyle{plain}
  \newtheorem{lem}[thm]{\protect\lemmaname}
  \theoremstyle{plain}
  \newtheorem{cor}[thm]{\protect\corollaryname}

\newref{prop}{name=Proposition~}
\newref{lem}{name=Lemma~}
\newref{thm}{name=Theorem~}
\newref{rem}{name=Remark~}
\newref{cor}{name=Corollary~}

\makeatother

\usepackage{babel}
  \addto\captionsbritish{\renewcommand{\corollaryname}{Corollary}}
  \addto\captionsbritish{\renewcommand{\definitionname}{Definition}}
  \addto\captionsbritish{\renewcommand{\examplename}{Example}}
  \addto\captionsbritish{\renewcommand{\lemmaname}{Lemma}}
  \addto\captionsbritish{\renewcommand{\propositionname}{Proposition}}
  \addto\captionsbritish{\renewcommand{\remarkname}{Remark}}
  \addto\captionsbritish{\renewcommand{\theoremname}{Theorem}}
  \addto\captionsenglish{\renewcommand{\corollaryname}{Corollary}}
  \addto\captionsenglish{\renewcommand{\definitionname}{Definition}}
  \addto\captionsenglish{\renewcommand{\examplename}{Example}}
  \addto\captionsenglish{\renewcommand{\lemmaname}{Lemma}}
  \addto\captionsenglish{\renewcommand{\propositionname}{Proposition}}
  \addto\captionsenglish{\renewcommand{\remarkname}{Remark}}
  \addto\captionsenglish{\renewcommand{\theoremname}{Theorem}}
  \providecommand{\corollaryname}{Corollary}
  \providecommand{\definitionname}{Definition}
  \providecommand{\examplename}{Example}
  \providecommand{\lemmaname}{Lemma}
  \providecommand{\propositionname}{Proposition}
  \providecommand{\remarkname}{Remark}
\providecommand{\theoremname}{Theorem}

\begin{document}

\title{The reals as rational Cauchy filters }

\author{Ittay Weiss}
\begin{abstract}
We present a detailed and elementary construction of the real numbers
from the rational numbers a la Bourbaki. The real numbers are defined
to be the set of all minimal Cauchy filters in $\mathbb{Q}$ (where
the Cauchy condition is defined in terms of the absolute value function
on $\mathbb{Q}$) and are proven directly, without employing any of
the techniques of uniform spaces, to form a complete ordered field.
The construction can be seen as a variant of Bachmann's construction
by means of nested rational intervals, allowing for a canonical choice
of representatives. 
\end{abstract}

\maketitle

\section{Introduction}

The aim of this article is to present an elementary construction of
the real numbers as a completion of the rational numbers following
Bourbaki's completion of a uniform space by means of minimal Cauchy
filters. Of the numerous ways of constructing the real numbers (see
\cite{WeissSurvey} for a survey, where the present construction is
outlined, christened the Bourbaki reals) perhaps the two most famous
approaches are Cantor's and Dedekind's. For reasons explained below,
we propose the present construction as a competitor in the categories
of elegance and of pedagogical importance to these two constructions. 

The construction of the real numbers we present is quickly motivated
in one of two ways. Bourbaki's approach to the real numbers (see \cite{bourbaki1998general})
is not to construct any particular model of the real numbers, but
rather to view them as a completion of $\mathbb{Q}$, viewed as a
uniform space. The proofs of the complete ordered field axioms are
facilitated through the use of general universal properties of the
completion, thus avoiding the technicalities of any particular construction.
While the elegance of this approach is undisputed it does not constitute
a proof of the consistency of the axioms of complete ordered fields
(relative to the rationals) without recourse to the existence of the
completion of a uniform space. Of course, Bourbaki also provides a
general construction of such a completion in terms of minimal Cauchy
filters. Thus, following Bourbaki and at the same time going against
Bourbaki's spirit of not constructing the reals, we do construct the
reals as minimal Cauchy filters of rational numbers. The aim of this
work is to present the details of this construction in a completely
elementary and self-contained fashion. 

A second motivation for the construction goes back to the late 19th
century and the early attempts of placing the real numbers on a rigorous
footing. Other than Cantor's construction by means of Cauchy sequences
and Dedekind's construction by cuts, some of the other attempts for
formal constructions of the real numbers were motivated by the properties
of nested intervals, though working out the details proved challenging,
finally culminating with Bachmann's construction (\cite{Bachmann1892}).
Some of the general interest in the potential of each of these three
alternatives is gathered from the following quote from \cite[page 46]{ebbinghaus1990numbers}
\begin{quote}
The practical advantages of nested intervals over cuts or fundamental
sequences are as follows. If the real number $x$ is described by
$(I_{n})$ the position of $x$ on the number axis is fixed within
defined bounds by each $I_{n}$. On the other hand with a fundamental
sequence $(r_{n})$, the knowledge of one $r_{n}$ still tells us
nothing about the position of $x$. Again, the description of $x$
as a cut $(\bar{\alpha},\beta)$ can result from a definition of the
set $\alpha$ by means of statements which say nothing directly about
the position of $x$. The theoretical disadvantage of using the nested
interval approach is that introducing the relation $\le$ between
equivalence classes of nets of nested intervals and verifying the
field properties for addition and multiplication is somewhat troublesome.
\end{quote}
Let us now note that any nest $\{I_{n}\}$ of rational intervals is
a Cauchy filter base and thus gives rise to a unique minimal Cauchy
filter of rationals. This observation gives a direct comparison between
Bachmann's construction and our construction, and in this sense our
construction can be seen as a variant of Bachmann's allowing for a
canonical choice of representatives. It is shown below (i.e., \propref{InternalApproximations})
that the practical advantages of Bachmann's construction are shared
with our construction, while the technical difficulties one encounters
with Bachmann's construction are confined in our construction to the
proof of one result (i.e., \lemref{sumOfReals}). The rest of the
construction is rather straightforward. Moreover, unlike in Bachmann's
construction, in our construction the definition of the ordering of
the reals, and the related proofs, are quite elegant. We thus hope
to place Bachmann's construction, through the variant we present,
as a competitor of potentially equal popularity as either Cantor's
construction or Dedekind's construction. 

The real numbers are defined below to be the set of all minimal Cauchy
filters in $\mathbb{Q}$. The ordering on the reals is given as follows.
For real numbers $a$ and $b$, we declare that 
\[
a<b
\]
precisely when there exist $A\in a$ and $B\in b$ such that 
\[
A<B
\]
universally, i.e., when 
\[
\alpha<\beta
\]
for all $\alpha\in A$ and $\beta\in B$. Equivalently, $a\le b$
holds precisely when for all $A\in a$ and $B\in b$
\[
A\le B
\]
existentially, i.e., when 
\[
\alpha\le\beta
\]
for some $\alpha\in A$ and some $\beta\in B$. The algebraic structure
on $\mathbb{R}$ is defined as follows. Given real numbers $a$ and
$b$, their \emph{sum }
\[
a+b
\]
is the filter generated by the filter base 
\[
\{A+B\mid A\in a,B\in b\},
\]
where
\[
A+B=\{\alpha+\beta\mid\alpha\in A,\beta\in B\}.
\]
Similarly, the \emph{product}
\[
ab
\]
is the filter generated by the filter base
\[
\{AB\mid A\in a,B\in b\},
\]
where 
\[
AB=\{\alpha\beta\mid\alpha\in A,\beta\in B\}.
\]
Below we give a detailed proof that the reals thus defined form a
complete ordered field, without a-priori use of uniform structures
or the completion process by means of minimal Cauchy filters. Thus
the treatment is completely elementary. 
\begin{rem}
A word on the originality content of this work is in order. Bourbaki's
construction of the reals as the uniform completion of the rationals
is certainly not new, nor is the use of minimal Cauchy filters in
the construction of the completion of any uniform space. Due to the
ambient well-developed general theory we find ourselves in the position
of being guaranteed that constructing the reals as minimal Cauchy
rational filters must work. However, the details of this construction
as we present below, other than being elementary, are not just the
result of unpacking the classical Bourbaki proofs. The definitions
of addition and multiplication of real numbers are given explicitly
on the level of the minimal Cauchy filters without the use of the
roundification process of a filter. The ordering structure, which
Bourbaki gives in terms of differences and positives, is also given
directly in terms of the minimal Cauchy filters (and this is perhaps
the main contribution of this work in terms of originality). In particular,
the construction lends itself to investigations of its usefulness
for implementation on a computer and, from a topos theoretic point
of view, the construction presents a new way to construct a real numbers
object in a topos. These issues are not investigated here. 
\end{rem}

\subsection{Plan of the paper}

The construction of the reals is given in \secref{Constructing-the-reals}
where we take as given a model $\mathbb{Q}$ for the rationals as
an ordered archimedean field. Since the prerequisites for the construction
are very modest, \secref{Preliminary-notions} is a self-contained
preliminary section giving an elementary treatment of the geometry
of intervals and the basics of filters. After the construction is
dealt with, \secref{Consequences} is a short presentation of two
consequences of the formalism - a proof of the uncountability of the
reals and a criterion for convergence. For the sake of presenting
the reader with a broader spectrum than just the details of the construction,
\secref{RoleOfReals} is a journey to some of the realms of modern
mathematics inspired by the real numbers.

\section{The real numbers and their role in shaping mathematics\label{sec:RoleOfReals}}

. In this section we provide a brief account of the real numbers,
of how the real numbers are typically modeled, and of the reciprocal
effects between the ambient mathematics and the desire to deepen our
understanding of the real numbers. Since a full treatment of these
issues can easily fill up an entire book, the material presented is
by necessity partial.

\subsection{Prehistory}

The ancient Greeks, and in particular the Pythagoreans, were very
fond of numbers. Alongside laying the foundations of modern axiomatic
rigor they also held magnificent superstitious beliefs about how numbers
relate to, and govern, nearly everything in the universe. The Greeks'
concept of a number was slightly different than our modern understanding
of what it is. While we accept number systems even if they are not
used (immediately) to measure anything real, the Pythagoreans were
highly motivated by geometry, and numbers were often used for, and
understood through, geometric interpretation. For many years the Greek
mathematicians and engineers were quite satisfied with their system
of ratios - a system quite close to the modern system of rational
numbers. The common belief was that ratios suffice for all practical
real needs, or that at the very least they suffice to measure all
geometric constructs precisely. It is thus that the famous discovery
(about which very few details are known) that $\sqrt{2}$ is irrational,
and consequently that the length of the hypotenuse of a right triangle
of side length $1$, a very ordinary and real object, can not be measured
as a ratio was received as a shock. 

The divide between rational and irrational real numbers had raised
the simple issue of just what is an irrational number, and so the
quest to find mathematical entities with which every reasonable measurement
is possible (even if only in theory) had begun. Somewhat astonishingly,
numerous centuries have passed before an answer was given in the late
19th century. The newly found models lay to rest the need for a definition
and allowed for the first time for a thorough investigation of the
real number line, and, quite unexpectedly, fundamental surprises with
significant ramifications were uncovered. 

During the prehistoric era of the real numbers, that is those days
following the Pythgoreans' discovery of the irrationality of $\sqrt{2}$
but preceding any formal construction of the reals, scientists had
to cope with the reality of an ever growing list of irrational numbers
(to which $e$ was added in 1737 by Euler and $\pi$ in 1761 by Lambert),
the ever increasing difficulty in discerning between the rational
and the irrational (it is still unknown for instance whether $e+\pi$
is rational or not), and the constant feeling of incompleteness due
to the fact that no mathematical system has been found yet that truly
captures the real line. Of course, that did not stop science from
progressing. After all, mathematics is merely a modeling tool for
the working scientist and as long as the rough idea is good enough
to work with, one does not necessarily need be discouraged by the
lack of rigorous details. 

Newton and Leibniz certainly were not deterred by the non-existent
foundations when they developed calculus. They employed not only the
real numbers but also infinitesimals - elusive numbers having the
property of being positive, yet smaller than any number of the form
$1/n$, for all $n\ge1$. The scientific revolution embodied in the
work of Newton and Leibniz drew immense attention from the scientific
community, among which the famous quote from Berkeley's ``The Analyst''
regarding infinitesimals as the ghosts of departed quantities, expressing
the growing discomfort at the lack of rigorous foundations. For the
first time in the prehistory of the real number system significant
disagreement was encountered on what did constitute a real number,
and what did not. It can be said with a fair amount of certainty that
the growing need to provide rigorous foundations for calculus and
the increased mathematical sophistication spawned by calculus had
a decisive role in the discovery of rigorous models of the real number
system.

\subsection{Constructions of the real numbers}

Prominent figures such as Bolzano and Weierstrass attempted to construct
the real numbers, with only partial success. The first correct models,
and also the most commonly presented constructions in modern textbooks,
were given by Dedekind (1872) and by Cantor (1873) using, respectively,
cuts of rational numbers and Cauchy sequences of rational numbers.
Bachmann's less well-known construction, employing nests of rational
intervals, was given around the same time (1892). This trio of constructions
is indicative of the fact that the major obstacle for previous generations
in obtaining a rigorous definition of the real numbers was technological
- the mathematical tools and techniques of calculus paved the way
for three mathematicians to come up with three different solutions
to the same problem. 

At long last, then, the real numbers were born (nearly as triplets)
and for a long while, about a century, no other constructions were
given. Of course, there was no pressing need for more constructions,
but nonetheless from 1960 till today some 16 other constructions have
been given (excluding the one presented in this work). The interested
reader is referred to \cite{WeissSurvey} for a comprehensive survey
of most, if not all, constructions of the real numbers found in the
literature. 

Modern textbooks often describe the real numbers axiomatically, simply
by listing the axioms of a complete ordered field. It is not hard
at all to prove that any two structures satisfying these axioms are
isomorphic. In other words, the theory of complete ordered fields
(which is a second order theory) is categorical. The axiomatic approach
is convenient enough to develop all of calculus and thus, in a sense,
the sole purpose of exhibiting an actual model of the axioms is to
ease one's suspicions (if any) that perhaps a contradiction is lurking
underneath the surface. It is seldom the case that one uses the particularities
of any given model in order to actually prove anything of interest.
Once the axioms are verified, the details of the construction have
served their purpose in establishing the relative consistency of the
axioms, and are promptly forgotten. 

The consistent insistence of producing more and more models of the
real numbers may thus be puzzling, and perhaps is atestment to human
curiosity more than anything else. The real numbers are so fundamental
that our fascination with them is not so easily quenched. Another
aspect is that the details of a particular constructions may be at
considerable odds with one's own intuition of what the real numbers
\emph{really }are that the construction may be considered flawed (on
some meta-mathematical level). One additional criterion for the success
of any particular construction is pedagogical. Given that any construction
of the reals is likely to require a non-negligible amount of time
to comprehend, the techniques employed and the details that must be
worked out had better be helpful to the student rather than form a
diversion from the main results and techniques, so as not to become
a waste of precious time. The review below takes a critical look at
Cantor's and Dedekind's constructions.

\subsubsection*{Cantor's construction}

Cantor presented his construction of the real numbers by means of
Cauchy sequences in \cite{cantor1966grundlagen}. We present here
the construction only.

Consider the collection $S$ of all rational sequences, i.e., all
sequences $(a_{n})_{n\ge1}$ where $a_{n}\in\mathbb{Q}$. Declare
such a sequence to be a \emph{null sequence} if its limit is $0$,
a condition given purely in terms of rational numbers by the condition
that for all rational $\varepsilon>0$ there exists $n_{0}\in\mathbb{N}$
such that $|a_{n}|<\varepsilon$ for all $n\ge n_{0}$. The relation
$(a_{n})\sim(b_{n})$ precisely when $(a_{n}-b_{n})$ is a null sequence,
is an equivalence relation on $S$, and the set of real numbers is
the quotient $S/{\sim}$. 

The construction quite evidently requires quotienting, though in a
rather standard form. However, it must be remembered that the student
encountering the reals in this fashion in a first rigorous analysis
course is not accustomed to the mental juggling of equivalence classes.
Consequently, even though the algebraic properties of the reals are
deduced quite straightforwardly from the corresponding properties
of $\mathbb{Q}$, pedagogically, it is questionable how effective
it is for the novice to be confronted with such a complicated apparatus,
whereby a real number is a set of sequences of rational numbers.

\subsubsection*{Dedekind's construction}

Dedekind presented his construction of the real numbers in terms of
sections, or cuts, of rational numbers in \cite{dedekind1930stetigkeit}.
Again, we present the construction itself, mentioning that few texts
actually go through the painful process of verifying all of the claims
required for validating the construction. Full details, spanning numerous
pages, can be found in Landau's \cite{Landau}, possibly the only
text to actually prove all of the details. 

A Dedekind cut $(L,R)$ consists of two non-empty sets $L$ and $R$
which partition the rational numbers, with $x<y$ for all $x\in L$
and $y\in R$. Every rational number determines two cuts; one where
$x$ is the largest element in $L$, and one where it is the smallest
element in $R$. To avoid double representations, an arbitrary choice
must be made: requiring $L$ does not have a largest element (or,
essentially equivalently, that $R$ does have have a smallest element).
The set of real numbers is then the set of all Dedekind cuts. 

The construction has a rather geometric flavour, addressing the incompleteness
of the rationals directly. Moreover, it is an easy exercise to explicitly
construct a Dedekind cut which does not correspond to any rational
number (e.g., $R=\{x\in\mathbb{Q}\mid x^{2}>2\}$), and so it is immediate
that one obtains new entities which were not there before (a similar
demonstration using Cantor's construction is somewhat contrived, and
prehaps less impressive, due to the algebraic nature of Cantor's construction
versus the geometric nature of Dedekind's). However, the amount of
detail required for a complete verification of the construction is
staggering. In a sense, the student is required to exchange one's
belief that a model of the real numbers exists by the belief that
a proof that Dedekind's construction is valid exists. Nobody really
expects anybody to go through the entire proof.

\subsubsection*{There really is nothing simple in the passage from $\mathbb{Q}$
to $\mathbb{R}$.}

The discussions above should make it clear that neither Cantor's nor
Dedekind's construction of the real numbers achieves the goal of introducing
the real numbers rigorously and palatably at the same time. The difficulties
present in one construction are somewhat complimentary to those in
the other construction, but each approach retains a considerable amount
of technical, conceptual, and pedagogical caveats. We conclude this
discussion by presenting further criticism of Cantor's and Dedekind's
constructions, voiced by Halmos and Conway. 

In \cite{halmos1985pure} Halmos reproachfully writes:
\begin{quotation}
``As far as Dedekind cuts are concerned, we abstractionists have
been arguing against them for a long time; it's not quite honest to
dump them in our laps and then accuse us of nurturing them. They are
a historical accident. Most students of mathematics learn them as
the first logically coherent way of constructing a complete ordered
field, but, so far as I know, they are out of fashion by now, or in
any event they ought to be. A Dedekind cut is a very narrowly focused
concept. It can be generalized to certain kinds of ordered sets, but
that possibility is of interest to specialists only. I am firmly convinced
that one can be a broadly cultured, creative mathematician without
knowing what a Dedekind cut is. Equivalence classes of Cauchy sequences
are easier to understand, and, three cheers, they are more algorithmic.
The important thing from the point of view of abstract mathematics,
however, is that sequences are ``cleaner'' than cuts, more widely
applicable, and more beautiful, and more structurally pertinent to
the study of analysis.''
\end{quotation}
We refer to Conway's words from \cite{Conway}, where the difficulties
inherent to Dedekind's construction are discussed:
\begin{quotation}
``In practice the main problem is to avoid tedious case discussions.
{[}Nobody can seriously pretend that he has ever discussed even eight
cases in such a theorem -- yet I have seen a presentation in which
one theorem actually had 64 cases!{]} Now if we define $\mathbb{R}$
in terms of Dedekind sections in $\mathbb{Q}$, then there are at
least four cases in the definition of the product $xy$ according
to the signs of $x$ and $y$ {[}And zero often requires special treatment!{]}.
This entails eight cases in the associative law $(xy)z=x(yz)$ and
strictly more in the distributive law $(x+y)z=xz+yz$ (since we must
consider the sign of $x+y$). Of course an elegant treatment will
manage to discuss several cases at once, but one has to work very
hard to find such a treatment.''
\end{quotation}
Shortly afterwards Conway provides the following criticism of Cantor's
construction, indicating its pedagogical difficulties: 
\begin{quotation}
``{[}The reader should be cautioned about difficulties in regarding
the construction of the reals as a particular case of the completion
of a metric space. If we take this line, we plainly must not start
by defining a metric space as one with a real-valued metric! So initially
we must allow only rational values for the metric. but then we are
faced with a problem that the metric on the completion must be allowed
to have arbitrary real values!

Of course, the problem here is not actually insoluble, the answer
being that the completion of a space whose metric takes values in
a field $\mathbb{F}$ is one whose metric takes values in the completion
of $\mathbb{F}$. But there are still sufficient problems in making
this approach coherent to make one feel that it is simpler to first
produce $\mathbb{R}$ from $\mathbb{Q}$, and later repeat the argument
when one comes to complete an arbitrary metric space, and of course
this destroys the economy of the approach. My own feeling is that
in any case the apparatus of Cauchy sequences is logically too complicated
for the simple passage from $\mathbb{Q}$ to $\mathbb{R}$ -- one
should surely wait until one has the real numbers before doing a piece
of analysis!{]}''
\end{quotation}
For a detailed account of the historical development surrounding these
constructions of the real numbers (and much more) the reader is referred
to \cite{adventures}. In the rest of this section we outline some
aspects of the interplay between the study of the real numbers and
modern mathematical developments, primarily analysis, set theory,
and logic.

\subsection{Transcendental numbers}

Advances in the techniques of analysis led to considerable achievements
on a superficially simple question regarding the nature of irrational
numbers. The first confirmed examples of irrational numbers, i.e.,
$\sqrt{2}$, $\sqrt{3}$, $\sqrt{5}$, etc. all belong to a family
of irrational numbers that are quite easy to verify. Namely, if $k\ge2$
is an integer and $n\ge2$ is an integer not of the form $m^{k}$,
$m\in\mathbb{N}$, then $\sqrt[k]{n}$ is irrational. The proof is
an easy consequence of the fundamental theorem of arithmetic. However,
every irrational number of the form $\sqrt[k]{n}$ satisfies an algebraic
equation with integer coefficients, namely $x^{k}-n=0$. Another family
of real numbers easily proved to be irrational are numbers of the
form $\log_{a}b$, for suitable values of $a$ and $b$. Again, these
numbers are easily shown to satisfy a polynomial relation with integer
coefficients. A real number $\alpha$ which is the root of a polynomial
with integer coefficients is called an \emph{algebraic }number. It
is immediate that any rational number is algebraic, the latter being
thus a natural \foreignlanguage{british}{generalisation} of the former.
The above examples illustrate that numbers that are easily shown to
be irrational tend to also be algebraic. The problem of existence
of transcendental numbers, i.e., non-algebraic numbers, was open until
1844 when \foreignlanguage{british}{Liouville} constructed the first
example of a transcendental number. 

Liouville's method is analytical and in fact produces a whole family
of transcendental numbers. It relies on the following lemma, which
is the heart of Liouville's construction.
\begin{lem*}[Liouville's Lemma]
If $\alpha$ is an irrational algebraic number which is the root
of a polynomial of degree $n>0$ with integer coefficients, then a
real number $L>0$ exists such that 
\[
|\alpha-\frac{p}{q}|>\frac{L}{q^{n}}
\]
holds for all integers $p,q$, with $q>0$. 
\end{lem*}
A proof, which is only of moderate difficulty, can be found in \cite{OxtobyMeasureCategory}.
The theorem may be interpreted as saying that an algebraic number
is either rational, or cannot be very well approximated by a rational.
It is thus immediate that any real number $\alpha$ for which there
exist sequences $(p_{n})_{n\ge1}$ and $(q_{n})_{n\ge1}$ of integers,
with $q_{n}\ge2$ for all $n\ge1$, and such that 
\[
|x-\frac{p_{n}}{q_{n}}|<\frac{1}{q_{n}^{n}}
\]
must be transcendental. Numbers satisfying this condition are called
\emph{Liouville numbers}. The existence of Liouville numbers, and
thus of transcendental numbers, is easily established, namely 
\[
\sum_{n=0}^{\infty}\frac{1}{10^{n!}}
\]
is a Liouville number. 

Liouville's result is of rather limited use when trying to settle
the status of naturally occurring numbers such as $e$ or $\pi$.
The proof that $e$ is transcendental was given in 1873 by Hermite,
while the fact that $\pi$ is transcendental was established in 1882
by Lindemann, both using more sophisticated techniques. 

The theory of measures of irrationality and transcendental number
theory carry these results much further. And yet, despite significant
advances, it is still unknown whether $\pi+e$ is transcendental or
not.

\subsection{The uncountability of $\mathbb{R}$ and the abundance of transcendental
numbers}

Numerous proofs of the uncountability of the real numbers are well-known
(see, e.g., \cite{baker2007uncountable} for an unorthodox proof)
and it is not the place here to repeat any of them (a proof utilizing
the construction presented in this work is given in \thmref{uncount}
below). Instead, we contemplate briefly the inherent difficulties
with infinite quantities, and concentrate on the impact of Cantor's
famous uncountability result. 

The concept of infinity posed dramatic challenges to some of the greatest
contributors to the development of science. Galileo in \emph{Two New
Sciences} presents his point of view on the matter through a discussion
involving Salviati, Simplicio, and Sagredo, where Simplicio, confronted
with some simple geometric observations about line segments, states
that:
\begin{quotation}
``Here a difficulty presents itself which seems to me insoluble.
Since it is clear that we may have one line greater than another,
each containing an infinite number of points, we are forced to admit
that, within one and the same class, we may have something greater
than infinity, because the infinity of points in the long line is
greater than the infinity of points in the short line. This assigning
to an infinite quantity a value greater than infinity is quite beyond
my comprehension.''
\end{quotation}
Interestingly, the dialogue continues with a discussion of positive
integers, with Salviati explaining to Simplicio that:
\begin{quotation}
``If I should ask further how many squares there are one might reply
truly that there are as many as the corresponding number of roots,
since every square has its own root and every root its own square,
while no square has more than one root and no root more than one square.''
\end{quotation}
in what is so remarkably close to the notion of cardinal equality
as well as to the proof that $\{n^{2}\mid n\in\mathbb{N}\}$ has the
same cardinality as $\mathbb{N}$. Salviati continues:
\begin{quotation}
``But if I inquire how many roots there are, it cannot be denied
that there are as many as there are numbers because every number is
a root of some square. This being granted we must say that there are
as many squares as there are numbers because they are just as numerous
as their roots, and all the numbers are roots. Yet at the outset we
said that there are many more numbers than squares, since the larger
portion of them are not squares. Not only so, but the proportionate
number of squares diminishes as we pass to larger numbers...''
\end{quotation}
leading Sagredo to ask:
\begin{quotation}
``What then must one conclude under those circumstances?''
\end{quotation}
to which Salviati tragically responds with:
\begin{quotation}
``So far as I can see we can only infer that the totality of all
numbers is infinite, that the number of squares is infinite, and that
the number of their roots is infinite; neither is the number of squares
less than the totality of all numbers, nor the latter greater than
the former; and finally the attributes ``equal'', ``greater'',
and ``less'', are not applicable to infinite, but only to the finite,
quantities.''
\end{quotation}
Such arguments are primarily used today by the lecturer, much to her
delight, in order to torment her students into acceptance of the formal
consequences resulting from the unavoidable notion that two sets between
which a bijection exists are equinumerous, at the cost of rejecting
one's false beliefs about infinity, rather than the other way around.
Galileo is unable to reconcile the facts and chooses to resolve the
situation by throwing the baby out with the bathwater - comparability
is only allowed for finite quantities. 

Cantor's creation of the haven of set theory began in 1874 with his
famous publication of an article in which he demonstrates the uncountability
of the reals, the countability of the algebraic numbers (which receives
much emphasis in the article), and thus concluding that transcendental
numbers exist without explicitly presenting any particular such number.
Opinions vary regarding the precise details of events leading to,
and following from, Cantor's seminal ideas. It is well-established
that Kronecker held a very narrow view on what is considered proper
mathematics, quite openly rejecting Cantor's results. In any case,
even if not stated quite so bluntly, the conclusion is that not only
do transcendental numbers exists (which was already shown by Liouville),
but that in a precise sense the vast majority of real numbers are
transcendental while a tiny proportion of real numbers are algebraic. 

Cantor's work finally made arguments about infinity possible. The
student coming to terms with the counterintuitive phenomena manifesting
infinite sets may take comfort in the fact that remnants of Galileo's
difficulties could still be found in Weierstrass' assertion (see \cite{ferreiros2008labyrinth}
for a much more thorough discussion) made in the summer of 1874, during
a course we gave, to the effect that 
\begin{quotation}
``two 'infinitely great magnitudes' are not comparable and can always
be regarded as equal, and that applying the notion of equality to
infinite magnitudes does not lead to any result.''
\end{quotation}
An assertion of equal counter progressive power as Galileo's conclusion
that only finite quantities can be compared. Cantor's insights freed
us from the shackles of such misconceptions.

\subsection{Other notions of the size of $\mathbb{R}$ - space filling curves}

Much of the counter intuitive nature of the fundamentals of cardinal
comparability, such as those discussed above, or the simply established
fact that $\mathbb{R}^{n}$ and $\mathbb{R}$ have the same cardinality,
obviously stems from the geometric extra baggage that the observer
brings with her when she perceives, e.g., $\mathbb{R}^{2}$ as a plane
versus viewing $\mathbb{R}$ as a line. As Cantor's work was being
digested by the mathematical community, some immediately sought a
more careful formulation of one's intuition that $\mathbb{R}$ is
considerably smaller than $\mathbb{R}^{2}$ by introducing topological
restrictions on the size comparison. There was no disputing that $[0,1]$
and $[0,1]\times[0,1]$ shared the same cardinality, but surely the
line segment $[0,1]$ can not be \emph{continuously} mapped to the
square $[0,1]\times[0,1]$ in such a way as to completely cover it. 

Directly motivated by Cantor's results, in 1890 Peano introduced the
first example showing that even this topological intuition is faulty
by constructing a surjective continuous function $[0,1]\to[0,1]\times[0,1]$,
a so called space-filling curve. In 1891 Hilbert constructed another
such curve. Faced with these results one must admit temporary defeat
in turning the 'obvious' fact that $\mathbb{R}$ has dimension $1$
and thus is significantly smaller then $\mathbb{R}^{2}$ whose dimension
is $2$, into a rigorous argument. Indeed, the topological notion
of dimension is, in light of the above, not at all surprisingly, a
subtle issue whose elucidation required considerable effort (see,
e.g., \cite{dimEngelking,dimPears}). Unfortunately, further discussion
here will take us too far afield from the main thread of this section.

\subsection{The continuum hypothesis}

Cantor's realization that the countable cardinality of $\mathbb{N}$
is strictly smaller than the cardinality $c$ (the continuum) of $\mathbb{R}$
immediately raises the question as to the existence of subsets $S\subseteq\mathbb{R}$
whose cardinality lies strictly between the countable and the continuum.
The standard formulation of that question is in the form known as
the continuum hypothesis, stating that any subset of $\mathbb{R}$
is either countable (finite included) or of cardinality $c$. The
story of the unexpected resolution of the continuum hypothesis is
the subject of numerous articles and books and we shall thus be very
brief. 

Cantor himself was quite frustrated by his inability to resolve the
situation despite many attempts to prove the continuum hypothesis
(it seems Cantor was convinced of its validity). By 1900, when Hilbert
addressed the mathematical community, the impact of Cantor's set theory
was firmly acknowledged and it was not at all unnatural that Hilbert
listed the continuum hypothesis as the first of the 23 problems aimed
at directing the efforts of mathematicians in the ensuing years. 

Interestingly, it is some astounding leaps in mathematical logic,
due primarily to work of Gödel leading to, and resulting from, his
negative answer to the second of Hilbert's 23 problems that prepared
the ground for the final resolution of the continuum hypothesis. Hilbert's
second problem calls for a finitistic proof of the consistency of
Peano's axioms of arithmetic. The impossibility of such a program
was demonstrated by Gödel in 1931 in the form of his famous incompleteness
theorem. 

At the time of nomination of the continuum hypothesis as the opening
problem in Hilbert's list, the theory of sets was still in its infancy.
In some sense, it was not even born yet; the axioms of set theory
were not yet formulated, as it was only in 1908 that Zermelo proposed
the first of several axiomatic systems, largely fueled by Hilbert's
address. With the rapid advances in logic in the first few decades
of the 20th century, Gödel was able to show in 1940 that the continuum
hypothesis can not be disproved from the very well-accepted Zermelo-Fraenkel
axioms of set theory (with or without the axiom of choice). It would
take another 23 years until Cohen proved in 1963 that the continuum
hypothesis can not be proved from the Zermelo-Fraenkel axioms either,
a result of tremendous importance and impact, which led to Cohen's
awarding of the Fields Medal in 1966. 

The continuum hypothesis is thus forever in limbo. Without a doubt
such a result was not suspected by Cantor, Hilbert, or any of their
contemporaries at the time the question emerged. It is wonderfully
astonishing that such a seemingly simple matter as determining the
nature of the cardinalities of subsets of $\mathbb{R}$ presented
a colossal challenge, served as fuel to much of the early development
of logic, and required the genius of two tremendous modern figures
to resolve. There is indeed nothing simple in the passage from $\mathbb{Q}$
to $\mathbb{R}$.

\subsection{The non-triviality of the concept of length}

Healthy geometric intuition dictates that the Riemann integral $\int_{0}^{1}f(x)dx$
of a function which is constantly $1$ except at finitely many points
must be equal to $1$. After all, a finite number of points is a negligible
amount when computing the area determined by the graph of a function,
and indeed the Riemann integral has the property that it is blind
to such minute changes. With the more refined understanding of infinities,
seeing that the rationals are countable while the reals are not, one
also expects the integral a function which is constantly $1$ on the
irrationals and constantly $0$ on the rationals to have integral
equal to $1$, for the exact same reason as above. However, a trivial
computation shows that the Riemann integral of that function does
not exist. Riemann's machinery is blind to finite changes, but it
is completely obliterated by infinite changes that ought to have no
effect. 

Lebesgue's theory of integration resulted from the need to repair
this (and other) deficiencies. The idea is beautifully simple and
profound, with unexpected ramifications. The Riemann integral is obtained
by introducing a partition of the $X$-axis, estimating the area bound
under the graph (let us assume all functions are non-negative) by
means of rectangles. For the Lebesgue integral one starts instead
with a partition of the $Y$-axis, and then pulling back each segment
to the $X$-axis by means of the inverse image under $f$ to obtain
a partition of the $X$-axis. Then each of these pre-images is used
as the base of a 'rectangle' in order to estimate the area under the
graph. Permitting $f$ to so interact more dynamically in the formation
of the partition of the $X$-axis, versus the more static approach
of the Riemann integral where the partition is forced upon the function,
suggests a process more finely tuned to the needs of the function,
and thus more likely to correctly capture the behaviour of more complicated
functions. The undisputed triumph of Lebesgue's theory of integration
is the result of the affirmation of this suggestion in a very broad
sense. 

However, before Lebesgue's integral can get off the ground, one must
cope with the need to measure the 'length' of the pre-images under
$f$. Such subsets of $\mathbb{R}$ can be quite wild, depending on
the function $f$, and in any case they need not look anything like
an interval, or even a union of intervals. Taking a step back, one
can formulate a simple question - indeed one that the ancient Greeks
could have entertained - namely how does one measure the length of
an arbitrary subset of $\mathbb{R}$, where \emph{length }is taken
in the sense of a notion that meaningfully extends the familiar length
of intervals. Of course, one must state the properties one expects
of such a length concept $\mu\colon\mathcal{P}(\mathbb{R})\to[0,\infty]$.
The following conditions are hard to object to:
\begin{itemize}
\item $\mu(\emptyset)=0$
\item $\mu([0,1])=1$
\item $\mu(\bigcup E_{i})=\sum\mu(E_{i})$, for all countable collections
$\{E_{i}\}_{i\ge1}$ of mutually disjoint subsets of $\mathbb{R}$
\item $\mu$ is translation invariant, meaning that $\mu(r+E)=\mu(E)$,
for all $E\subseteq\mathbb{R}$ and $r\in\mathbb{R}$, where $r+E=\{r+x\mid x\in E\}$. 
\end{itemize}
With such an assignment $\mu$ (hopefully uniquely determined) Lebesgue's
theory can carry through whereby all functions will be integrable.
However, the relentless tendency of $\mathbb{R}$ to harbor surprises
strikes again and this marvelous dream is shattered to countably many
pieces by Vitali's famous example illustrating the inconsistency of
the four axioms above. The construction is far from obvious, but its
details are very simple, given in the following sequence of exercises:
\begin{itemize}
\item Declare, for real numbers $x,y$, that $x\sim y$ when $x-y\in\mathbb{Q}$,
and prove this is an equivalence relation. 
\item Use the axiom of choice to construct a set $V\subseteq[0,1]$ consisting
of precisely one representative of each equivalence class $[x]$,
for each $x\in[0,1]$. 
\item For each rational number $q\in[0,1)$ let $V_{q}=\{a+q\mid a\in V\}$,
where the computation is done modulo $1$, so that $V_{q}$ is again
a subset of $[0,1]$. 
\item Note that the family $\{V_{q}\}_{q}$ is a countable (since the rational
are countable) partition of $[0,1]$ (by the definition of the equivalence
relation), and $\mu(V_{q})=\mu(V)$ (since $\mu$ is translation invariant).
\item It then follows that $1=\mu([0,1])=\sum_{q}\mu(V_{q})=\sum_{q}\mu(V)$.
\item Finally, $\mu(V)=0$ implies $\mu([0,1])=0$, while $\mu(V)>0$ implies
$\mu([0,1])=\infty$, leading in either case to a contradiction.
\end{itemize}
The inability to consistently measure the length of all subsets of
$\mathbb{R}$ has immediate consequences. To save Lebesgue's program,
one must restrict to a $\sigma$-algebra of measurable subsets of
$\mathbb{R}$, the so called Borel measurable sets. This $\sigma$-algebra
is the smallest sigma algebra containing the intervals. This unavoidable
complication is resolved quite adequately, where in fact it is seen
that the crucial properties of $\mathbb{R}$ guaranteeing everything
still ticks well enough is that it is a Polish space, namely a second
countable metrizable topological space whose topology can be induced
by a complete metric. The Borel $\sigma$-algebra is then the $\sigma$-algebra
generated by the open sets, leading to the Borel hierarchy, where
the position in the hierarchy of a given Borel measurable set is determined
by how many times one must perform the operations of countable unions
and complementations in order to obtain the set. The theory of Polish
spaces is of fundamental importance in probability theory, mathematical
statistics, and descriptive set theory.

\subsection{Ghosts (of departed quantities) are real!}

Infinitesimal numbers, those ghosts of departed quantities, apparently
truly dead and disposed off due to Cauchy's rigorous formalization
of the limit concept, still had a few tricks up their sleeves. In
a fantastic turn of events, the discovery of non-standard models of
the natural numbers paved the way for Robinson in the early 1960's
to use the axiom of choice and to revive the long abandoned infinitesimals,
promoting their status from ghosts to flesh and blood mathematical
entities. We quote from Robinson's book (\cite{Robinson}):
\begin{quotation}
``It is shown in this book that Leibniz's ideas can be fully vindicated
and that they lead to a novel and fruitful approach to classical Analysis
and to many other branches of mathematics. The key to our method is
provided by the detailed analysis of the relation between mathematical
languages and mathematical structures which lies at the bottom of
contemporary model theory.''
\end{quotation}
Robinson's original work introducing infinitesimals in a rigorous
fashion used a heavy dose of higher order logic and is quite demanding
to non-logicians. Since then  other approaches emerged, some of which,
due to the prevalent use of ultra-products in model theory, requiring
only a modest amount of preparation. The hyperreals for instance can
be presented with little difficulty assuming about as much knowledge
of filters as is required for the construction of the reals presented
below, and thus we feel it fits nicely to present it here. 

The existence of a non-principal ultrafilter on $\mathbb{N}$ follows
by a simple application of Zorn's lemma by extending the filter of
cofinite subsets of $\mathbb{N}$. Let us fix such a filter $\mathcal{F}$.
Let $S$ be the set of all sequences $(a_{n})$ or real numbers and
introduce the relation $(a_{n})\sim(b_{n})$ precisely when $\{n\in\mathbb{N}\mid a_{n}=b_{n}\}\in\mathcal{F}$.
Thus, thinking of $\mathcal{F}$ as allowing majority sets to filter
through, two sequences are considered equivalent if a majority of
indices sees them as equal. It is straightforward to verify that $\sim$
is an equivalence relation on $S$. The importance of $\mathcal{F}$
being non-principal is that this equivalency does not hinge on just
one fixed index (so this majority democratic system is not a dictatorship).
The set of hyperreals is then the quotient $\mathbb{H}=S/{\sim}$,
and the reals are identified therein as $[(a_{n})]$ where $(a_{n}=a)$
is a constant sequence. One can then quite straightforwardly prove
that $\mathbb{H}$ is a field and that the transfer principle applies:
every first order sentence is true in $\mathbb{R}$ if, and only if,
it is true in $\mathbb{H}$. A very detailed account is given in \cite{Gold}.
We mention that several other approaches to infinitesimals exist,
some of which were developed with the ambitious objective of replacing
the standard Cauchy formalism. Perhaps the most pedagogically trialled
of those is \cite{ultrasmall}. 

The hyperreal formalism makes it deceptively simple to present a flesh
and blood infinitesimal. Indeed, $h=[(1/n)_{n\ge1}]$ is positive
(since the set of indices where $1/n>0$ is $\mathbb{N}$, certainly
a majority set) while $h<1/m$ for any natural number $m$ (since
the set of indices where $1/n<1/m$ is cofinite, and thus again a
majority set). However, there is a price to pay for allowing infinitesimals
into the world of ordinary real numbers. Consider for instance the
element $[(a_{n})_{n\ge1}]=[(1,2,1,2,1,2,1,2,\ldots)]$. It is clearly
either equal to $1$ or to $2$ (since the set of indices $n$ where
the claim ``$a_{n}=1$ or $a_{n}=2$'' holds is $\mathbb{N}$),
but due to the non-constructive nature of the ultrafilter $\mathcal{F}$,
it is impossible to determine which one it is. To date, all rigorous
approaches producing systems where infinitesimals exist alongside
the real numbers appear to pay a similar price in some form or another.

\subsection{The tame side of the real numbers}

Many of the results mentioned above may seem to paint a pessimistic
picture of the state of affairs and our ability to understand the
real numbers. Results such as the undecidability of the continuum
hypothesis, the impossibility to consistently measure all subsets
of $\mathbb{R}$, the difficulty of discerning transcendental numbers
from algebraic ones, the immense complexity of the Borel hierarchy,
and other results, as negative as they may appear, are, after all,
simply the reflections of the true nature of the real numbers. It
is just a fact of life that $\mathbb{R}$ is so beautifully complicated. 

As awe inspiring as the complexity of $\mathbb{R}$ is, some aspects
of it must be simpler than others and it seems pertinent to identify
this tamer side of the picture. A natural starting point, at least
from the point of view of model theory, is to look at the theory of
$\mathbb{R}$, namely all sentences (in some fixed language) which
are true in $\mathbb{R}$. The choice of language may, of course,
wildly change the theory, since languages with more symbols have greater
expressive power. Of particular importance is whether or not quantification
is restriced to only be allowed over elements (i.e., first order logic)
or if it is also allowed over sets and functions (i.e., second order
logic). The second order theory of $\mathbb{R}$ as a field is categorical,
namely any two models (within the same ambient model of sets of course)
which satisfy the same second order sentences that $\mathbb{R}$ does
is isomorphic to $\mathbb{R}$, and so the theory completely captures
the model. 

On the other hand, the first order theory of $\mathbb{R}$ as a field
is not categorical. A \emph{real closed field }is a field $\mathbb{F}$
sharing the same first order theory with $\mathbb{R}$ as a field,
namely a first order sentence (formulated, e.g., in the language for
rings) is true in $\mathbb{F}$ if, and only if, it is true in $\mathbb{R}$.
The field $\mathbb{H}$ of hyperreals is an example of a real closed
field. After all, by design, it has the same first order theory as
$\mathbb{R}$. Obviously though, the two structures are not isomorphic,
and so, indeed, the first order theory of the field $\mathbb{R}$
is not categorical. In 1926 Artin and Schreier proved that any ordered
field admits an essentially unique order field extension which is
a real closed field. In a sense then the first order theory of the
field $\mathbb{R}$ governs all ordered fields.

Among the developments in model theory in the first half of the 20th
century was the concept of quantifier elimination. A theory admits
quantifier elimination if any formula is equivalent to one in which
no quantifiers appear. Of course, it is remarkable when a theory admits
quantifier elimination and it carries important consequences. Some
form of quantifier elimination in $\mathbb{R}$ is well-known. For
instance, the fact that a quadratic with real coefficients admits
real roots if, and only if, its descriminant is non-negative can be
restated as follows. The first order formula $\varphi(a,b,c)=(\exists x\quad ax^{2}+bx+c=0]$
is equivalent to the first order formula $\psi(a,b,c)=[b^{2}-4ac\ge0]$,
a formula without quantifiers. Similarly, the first order formula
in $2n+2$ variables stating that a polynomial of degree $2n+1$ with
real coefficients has a real root is equivalent to any formula expressing
a tautology. 

Each of the examples above, where a formula was replaced by an equivalent
quantifier free one, involves the ordering on $\mathbb{R}$, and this
turns out not to be coincedental. In 1951 Tarski proved that ${\rm Th}(\mathbb{R})$,
the first order theory of $\mathbb{R}$ as an \emph{ordered field}
admits quantifier elimination. To be clear then, each and every first
order formula $\varphi$ in a language suited to speak of ordered
fields is equivalent to a quantifier free first order formula $\psi$
in the sense that in any model of ${\rm Th}(\mathbb{R})$ the formula
$\varphi$ holds if, and only if, $\psi$ holds. The proof is in fact
algorithmic; it describes precisely how to construct the quantifier
free formula $\psi$. However, of course, when applied to any of the
formulas from the previous paragraph, the resulting quantifier free
formula will be rather cumbersome. 

As a result of quantifier elimination it follows that ${\rm Th}(\mathbb{R})$
is complete, decidable, and $o$-minimal. The completeness of the
theory is the statement that for every first order sentence $\varphi$,
either it or its negation is provable, and so a situation like that
presented by the continuum hypothesis is ruled out. The decidability
of ${\rm Th}(\mathbb{R})$ is the claim that there exists an algorithm
that decides in finite time for any given first order sentence if
it is true or not (of course, there is no guarantee that such a process
is effective - in fact, all of the known algorithms are of immensely
high computational complexity). Finally, $o$-minimality implies that
every definable set $S\subseteq\mathbb{R}$ (i.e., one which corresponds
to the set of points satisfying a first order formula) is a finite
union of intervals (points included).

\subsection{The real numbers and digital computers}

Finally, we touch upon the interface between the real numbers and
digital computers. Turing already pioneered the notion of computable
numbers. Somewhat informally, a computable number is a real number
$a$ for which there exists an algorithm which for each $n\ge1$ produces
the first $n$ digits of $a$. Arguably, the computable numbers are
the \emph{real }real numbers; after all, if the digits of a number
can't ever be computed, does it really exist? We shall not delve into
this question. Instead, we passively note some of the properties of
the computable numbers. Firstly, they do form a field, obviously a
countable one, and in fact a real closed field, and thus they share
the same first order theory with the reals. However, the ordering
on the computable reals is not computable. Indeed, arguing informally,
suppose a decision process exists which decided in finite time for
each real number $a$ whether or not $a>0$. This algorithm must use
the machine describing the computable number $a$, look at the resulting
digits, and come to a conclusion. However, it is impossible to know
in advance how many digits are required to distinguish between the
number $0$ and a number $\varepsilon$ lying very close to $0$.
The machine to describe each will output, when given a value $n$,
a long list of $0$'s, unless $n$ is sufficiently large. But there
is no way to a-priori know how large $n$ should be. 

Such arguments and phenomena are similar to what one encounters in
constructive mathematics. Famously, constructivists reject the claim
that $P\vee\neg P$ is a tautology. The demand on proofs that a constructivist
places require any argument leading to the conclusion that a certain
object exists, to actually provide a method of constructing that object.
In particular, deducing something exists simply because entertaining
the assertion of its non-existence leads to a contradiction is not
an acceptable constructive proof of existence. Taking a constructive
approach to mathematics is one the one hand quite natural, particularly
from the perspective of computability and implementability on a computer.
But on the other hand the resulting mathematics has certain alien
attributes (undoubtedly rooted in the fact that mainstream mathematics
is classica, not constructive) and tends to be quite confusing. For
instance, the classical proof of the irrationality of $\sqrt{2}$
proceeds by contradiction, and follows through with some arithmetic
modulo $2$. This proof however is constructively valid, since the
definition of irrationality is inherently negative: a number is irrational
if it is not rational. The classical proof is a direct demonstration
that indee $\sqrt{2}$ is not rational. However, consider the claim
that there exists two irrational numbers $a,b$ such that $a^{b}$
is rational. To prove the assertion, consider the pair $a=\sqrt{2}$
and $b=\sqrt{2}$, or the pair $c=\sqrt{2}^{\sqrt{2}}$ and $d=\sqrt{2}$.
If $a^{b}$ is irrational, then a suitable pair is found. Elementary
considertaions show that if $a^{b}$ is irrational, then the pair
$c,d$ is then a suitable pair. In either way then, either the pair
$a,b$ or the pair $c,d$ is the desired pair. However, the proof
does not point to which one of the two pairs satisfies the claim.
This proof is not constructively valid. We refer the interested reader
to \cite{bridges1994constructive,harrison1994constructing,harrison1998theorem}
for further details on a constructive view of analysis and on the
real numbers from a computer oriented perspective.

\section{Preliminary notions\label{sec:Preliminary-notions}}

We collect here some background facts about intervals and  filters
which are used in the construction of the real numbers in the following
section. At this point we set the convention for the rest of this
work that the symbols $\varepsilon$, $\delta$, $\eta$, as in $\varepsilon>0$,
stand for rational numbers.

\subsection{The geometry of intervals}

We list elementary and easily verified geometric properties of intervals
in $\mathbb{Q}$ (which also hold for intervals in $\mathbb{R}$).
By an \emph{interval} $I$ we mean a subset of $\mathbb{Q}$ of the
form $(a,b)=\{x\in\mathbb{Q}\mid a<x<b\}$, where $a,b\in\mathbb{Q}$
with $a<b$. The \emph{length} of $I=(a,b)$ is $b-a$. Given $p\in\mathbb{Q}$
and a rational $\varepsilon>0$ we denote the interval $(p-\varepsilon,p+\varepsilon)$
by $p_{\varepsilon}$. We note that the addition and multiplication
operations of rational numbers extend element wise to addition and
multiplication operations on intervals. In more detail, given intervals
$I$ and $J$ we define 
\[
I+J=\{x+y\mid x\in I,y\in J\}
\]
and 
\[
I\cdot J=IJ=\{xy\mid x\in I,y\in J\}.
\]
We are interested in the geometric effect of these operations, some
of which are listed, without proof, below. 
\begin{prop}
The following properties are easily verified:\label{prop:GeomOfIntervals}\end{prop}
\begin{enumerate}
\item The sum of two intervals is again an interval. In fact, $p_{\varepsilon}+q_{\delta}=(p+q)_{\varepsilon+\delta}$,
for all $p,q\in\mathbb{Q}$, $\varepsilon>0$, and $\delta>0$.
\item The product of two intervals is again an interval.
\item Given any $p\in\mathbb{Q}$ and $\varepsilon>0$ one has $p_{\varepsilon}\subseteq y_{2\varepsilon}$
for all $y\in p_{\varepsilon}$. 
\item Given an interval $q_{\varepsilon}$ with $q-\varepsilon>0$ (respectively
$q+\varepsilon<0$) the set $\frac{1}{q_{\varepsilon}}=\{\frac{1}{x}\mid x\in q_{\varepsilon}\}$
is again an interval whose length is $\frac{2\varepsilon}{q^{2}-\varepsilon^{2}}$. 
\end{enumerate}

\subsection{Filters\label{sub:Filters}}

The following well-known notions and facts are stated for filters
in $\mathbb{Q}$ but hold verbatim in any metric space (where the
absolute value is replaced by the distance function, and the restriction
on $\varepsilon$ being rational is replaced by it being a real number).
The main aim of this subsection is to establish that any proper Cauchy
filter contains a unique minimal Cauchy filter.

\subsubsection{Filters and bases}

A \emph{rational filter }or more simply a \emph{filter }(since we
will only consider rational filters) is a non-empty collection $\mathcal{F}$
of subsets of $\mathbb{Q}$ such that 
\begin{itemize}
\item $F_{1},F_{2}\in\mathcal{F}$ implies $F_{1}\cap F_{2}\in\mathcal{F}$
\item $F_{2}\supseteq F_{1}\in\mathcal{F}$ implies $F_{2}\in\mathcal{F}$
\end{itemize}
hold for all $F_{1},F_{2}\subseteq\mathbb{Q}$. The second condition
implies that the only filter containing $\emptyset$ is the filter
$\mathcal{P}(\mathbb{Q})$, called the \emph{improper filter}. It
is very easy to verify that the intersection of any family of filters
is again a filter. 

Quite often, describing a filter is facilitated by considering only
part of a filter, and then adding necessary subsets to it to form
a filter. In more detail, a \emph{filter base }is a non-empty collection
$\mathcal{B}$ of subsets of $\mathbb{Q}$ such that for all $B_{1},B_{2}\in\mathcal{B}$
there exists $B_{3}\in\mathcal{B}$ with $B_{3}\subseteq B_{1}\cap B_{2}$.
Obviously, any filter is a filter base, but not vise versa. Given
a filter base $\mathcal{B}$ the collection $\langle\mathcal{B}\rangle=\{F\subseteq\mathbb{Q}\mid F\supseteq B,B\in\mathcal{B}\}$
is easily seen to be a filter. In fact, it is the smallest filter
containing $\mathcal{B}$ and is called the filter \emph{generated
}by the filter base $\mathcal{B}$.

\subsubsection{Cauchy and round filters}

The collection of all rational filters is large and varied. We will
be interested primarily in filters that, in a sense, are \emph{concentrated}.
The precise condition is called the \emph{Cauchy condition}, stating
that for every rational $\varepsilon>0$ there exists a rational number
$q\in\mathbb{Q}$ with $q_{\varepsilon}\in\mathcal{F}$. Any such
filter $\mathcal{F}$ is called a Cauchy filter. It is easy to see
that the Cauchy condition can equivalently be reformulated as follows.
For every rational $\varepsilon>0$ there exists a rational interval
$I\in\mathcal{F}$ whose length does not exceed $\varepsilon$. These
conditions will be used interchangeably according to convenience.

The Cauchy condition on a filter can be detected on a filter base
for it, as follows. Say that a filter base $\mathcal{B}$ satisfies
the Cauchy condition if for every rational $\varepsilon>0$ there
exist $q\in\mathbb{Q}$ and $B\in\mathcal{B}$ with $q_{\varepsilon}\supseteq B$.
We then say that $\mathcal{B}$ is a \emph{Cauchy filter base}. It
is straightforward to verify that if $\mathcal{B}$ is a Cauchy filter
base, then $\langle\mathcal{B}\rangle$ is a Cauchy filter. 

Given two filters $\mathcal{F}$ and $\mathcal{G}$, it is said that
$\mathcal{G}$ \emph{refines }$\mathcal{F}$ if $\mathcal{G}\supseteq\mathcal{F}$.
It is a trivial observation that if $\mathcal{G}$ refines $\mathcal{F}$
and $\mathcal{F}$ is Cauchy, then $\mathcal{G}$ is Cauchy as well,
and this shows that there is no point in asking for a unique maximal
Cauchy filter containing a given Cauchy filter. However, the converse
situation, i.e., asking for a minimal Cauchy filter contained in a
given Cauchy filter, is very interesting. To be precise, a \emph{minimal
Cauchy filter }is a Cauchy filter $\mathcal{F}$ such that if $\mathcal{G}$
is any Cauchy filter satisfying $\mathcal{G}\subseteq\mathcal{F}$,
then $\mathcal{G}=\mathcal{F}$. Before we can show that any Cauchy
filter contains a unique minimal Cauchy filter we need to introduce
the concept of a round filter. 

A filter $\mathcal{F}$ is \emph{round }if for every $F\in\mathcal{F}$
there exists a rational number $\varepsilon>0$ such that for all
$q\in\mathbb{Q}$ if $q_{\varepsilon}\in\mathcal{F}$, then $q_{\varepsilon}\subseteq F$.
Equivalently, $\mathcal{F}$ is round if for every $F\in\mathcal{F}$
there exists a rational number $\varepsilon>0$ such that any interval
$I\in\mathcal{F}$ of length not exceeding $\varepsilon$ satisfies
$I\subseteq F$. 

Roundness can also be detected on bases, as follows. Say that a filter
base $\mathcal{B}$ is a \emph{round filter base }if for every $B\in\mathcal{B}$
there exists $\varepsilon>0$ such that if $q_{\varepsilon}\supseteq B'$
for some $B'\in\mathcal{B}$, then $q_{\varepsilon}\subseteq B$.
It is immediate that if $\mathcal{B}$ is a round filter base, then
$\langle\mathcal{B}\rangle$ is a round filter. 
\begin{example}
The improper filter $\mathcal{P}(\mathbb{Q})$ was already remarked
to be Cauchy. It is not round since obviously the roundness condition
for $F=\emptyset$ can not be met. Further, fix a rational number
$q\in\mathbb{Q}$ and consider the collection $\{F\subseteq\mathbb{Q}\mid q\in F\}$
(which is the filter generated by the filter base $\{\{q\}\}$) and
the collection $\{F\subseteq\mathbb{Q}\mid q_{\varepsilon}\subseteq F,\varepsilon>0\}$
(which is the filter generated by the filter base $\{q_{\varepsilon}\mid\varepsilon>0,\varepsilon\in\mathbb{Q}\}$).
Both filters are Cauchy filters, but, the reader may verify, the latter
is round while the former is not. In fact, the latter is the unique
minimal Cauchy filter contained in the former. 
\end{example}
We shall see that it is no coincidence that the conjunction of the
Cauchy and roundness conditions amounts to minimal Cauchy. In fact,
we can immediately obtain the following result. 
\begin{prop}
If $\mathcal{F}$ is Cauchy and round, then $\mathcal{F}$ is minimal
Cauchy. \label{prop:CauchyRoundImpliesMin}\end{prop}
\begin{proof}
Assume that $\mathcal{G}\subseteq\mathcal{F}$ is a Cauchy filter.
We need to show that $\mathcal{G}=\mathcal{F}$, thus let $F\in\mathcal{F}$
be arbitrary. Since $\mathcal{F}$ is round there exists $\varepsilon>0$
such that if $q_{\varepsilon}\in\mathcal{F}$, then $q_{\varepsilon}\subseteq F$.
Now, since $\mathcal{G}$ is Cauchy there exists $q\in\mathbb{Q}$
with $q_{\varepsilon}\in\mathcal{G}$, and thus $q_{\varepsilon}\in\mathcal{F}$.
We conclude that $q_{\varepsilon}\subseteq F$ and thus, by the second
condition defining a filter, that $F\in\mathcal{G}$, as required. 
\end{proof}
Proving the converse requires a bit more work.

\subsubsection{Roundification}

A proper filter which is not round can canonically be sifted to yield
a round filter. The details are as follows. Given a rational $\varepsilon>0$
and a subset $F\subseteq\mathbb{Q}$, let $F_{\varepsilon}=\{x\in\mathbb{Q}\mid|x-y|<\varepsilon,y\in F\}$.
If $\mathcal{F}$ is a filter, then it is a simple matter to check
that $\{F_{\varepsilon}\mid F\in\mathcal{F},\varepsilon>0\}$ is a
filter base. The filter generated by that filter base is denoted by
$\mathcal{F}_{\circ}$ and is called the \emph{roundification }of
$\mathcal{F}$. Notice that $F\subseteq F_{\varepsilon}$, and thus
$\mathcal{F}_{\circ}\subseteq\mathcal{F}$. The following result justifies
the terminology. 
\begin{prop}
If $\mathcal{F}$ is a proper filter, then $\mathcal{F}_{\circ}$
is a round filter. \end{prop}
\begin{proof}
Let $G\in\mathcal{F}_{\circ}$ be given, i.e., $G\supseteq F_{\varepsilon}$
for some $F\in\mathcal{F}$ and $\varepsilon>0$. It now suffices
to find a $\delta>0$ such that if $q_{\delta}\in\mathcal{F}_{\circ}$,
then $q_{\delta}\subseteq F_{\varepsilon}$. Consider $\delta=\frac{\varepsilon}{2}$,
and suppose $q_{\delta}\in\mathcal{F}_{\circ}$, i.e., $q_{\delta}\supseteq F'$
for some $F'\in\mathcal{F}$. Since $\mathcal{F}$ is proper it follows
that $F\cap F'\ne\emptyset$. Using any element $y\in F\cap F'$ it
is now elementary that $q_{\delta}\subseteq F_{\varepsilon}$, as
required. \end{proof}
\begin{prop}
If $\mathcal{F}$ is a Cauchy filter, then $\mathcal{F}_{\circ}$
is Cauchy as well. \end{prop}
\begin{proof}
It suffices to show that $\mathcal{B}=\{F_{\varepsilon}\mid F\in\mathcal{F},\varepsilon>0\}$
is a Cauchy filter base, and thus let $\varepsilon>0$ be given. Since
$\mathcal{F}$ is Cauchy there exists $q\in\mathbb{Q}$ with $q_{\frac{\varepsilon}{2}}\in\mathcal{F}$.
Since $q_{\varepsilon}=(q_{\frac{\varepsilon}{2}})_{\frac{\varepsilon}{2}}\in\mathcal{B}$
the proof is complete. 
\end{proof}
We can now establish the converse of \propref{CauchyRoundImpliesMin}.
\begin{prop}
If $\mathcal{F}$ is minimal Cauchy, then $\mathcal{F}$ is round
and $\mathcal{F}_{\circ}=\mathcal{F}$. \end{prop}
\begin{proof}
Since $\mathcal{F}$ is minimal Cauchy $\mathcal{F}$ is proper, and
thus the filter $\mathcal{F}_{\circ}$ is round and Cauchy. But $\mathcal{F}_{\circ}\subseteq\mathcal{F}$,
and so the minimality of $\mathcal{F}$ implies $\mbox{\ensuremath{\mathcal{F}}}=\mathcal{F}_{\circ}$,
a round filter. 
\end{proof}
We emphasize thus that we just established that a filter is minimal
Cauchy if, and only if, it is Cauchy and round. 
\begin{thm}
If $\mathcal{F}$ is a proper Cauchy filter, then $\mathcal{F}_{\circ}$
is the unique minimal Cauchy filter contained in $\mathcal{F}$. \end{thm}
\begin{proof}
Since $\mathcal{F}$ is proper and Cauchy it follows that $\mathcal{F}_{\circ}$
is both Cauchy and round, and thus minimal Cauchy. Suppose now that
$\mathcal{G}\subseteq\mathcal{F}$ is some minimal Cauchy filter.
Applying the roundification process, which clearly preserves set inclusion,
yields $\mathcal{G}_{\circ}\subseteq\mathcal{F}_{\circ}$. But $\mathcal{G}_{\circ}=\mathcal{G}$
(since $\mathcal{G}$ is already minimal Cauchy) and the minimality
condition now implies that $\mathcal{G}=\mathcal{F}_{\circ}$, and
thus $\mathcal{F}_{\circ}$ is the only minimal Cauchy filter contained
in $\mathcal{F}$. \end{proof}
\begin{rem}
The results presented above are completely standard. For a categorical
perspective, exhibiting the roundification process as a left adjoint,
see \cite{chand2015completion}.
\end{rem}

\section{Constructing the reals\label{sec:Constructing-the-reals}}

We now present the real numbers.

\subsection{The set of real numbers }

Having laid down the filter theoretic preliminaries in \subref{Filters}
we immediately proceed with the definition of the set of real numbers. 
\begin{defn}
Let $\mathbb{R}$ denote the set of all minimal rational Cauchy filters.
Elements of $\mathbb{R}$ are called \emph{real numbers }and are typically
denoted by $a,b,c$. In particular, each real number $a$ is a collection
of subsets of rational numbers, and we will typically refer to these
sets by writing $A\in a$, while typical elements of $A$ will be
denoted by $\alpha\in A$. 
\end{defn}
We emphasize that a real number $a$ is necessarily a proper filter,
and thus $\emptyset\notin a$ and consequently $A\cap A'\ne\emptyset$
for all $A,A'\in a$. Moreover, if $A\in a$, then there exists $q\in\mathbb{Q}$
and $\varepsilon>0$ with $q_{\varepsilon}\in a$ and $q_{\varepsilon}\subseteq A$.
Indeed, as minimal Cauchy filters coincide with Cauchy and round filters,
$a$ must be round and so there exists $\varepsilon>0$ such that
$q_{\varepsilon}\in a$ implies $q_{\varepsilon}\subseteq A$, for
all $q\in\mathbb{Q}$. Since $a$ is also Cauchy, at least one $q\in\mathbb{Q}$
with $q_{\varepsilon}\in a$ does exist.

\subsection{Order}

The usual ordering of the rationals extends in two ways to the set
$\mathcal{P}(\mathbb{Q})$ of all subsets of $\mathbb{Q}$, namely
universally and existentially. Given subsets $A,B\subseteq\mathbb{Q}$
we write 
\[
A\le_{\forall}B
\]
if 
\[
\forall\alpha\in A,\beta\in B\colon\quad\alpha\le\beta.
\]
Similarly, we write 
\[
A\le_{\exists}B
\]
if 
\[
\exists\alpha\in A,\beta\in B\colon\quad\alpha\le\beta.
\]
The meaning of $A<_{\forall}B$ and $A<_{\exists}B$, as well as $A>_{\exists}B$
and $A>_{\forall}B$ etc., is defined along the same lines. Notice
thus that the negation of, for instance, $A<_{\exists}B$ is $B\ge_{\forall}A$.
Important to the results below is the following trivial observation,
whose proof is thus omitted. 
\begin{prop}
\label{prop:Transitivity}The relation $<_{\forall}$ is transitive
on $\mathcal{P}(\mathbb{Q})\setminus\{\emptyset\}$. In more detail,
$A<_{\forall}B<_{\forall}C$ implies $A<_{\forall}C$, for all non-empty
$A,B,C\subseteq\mathbb{Q}$. 
\end{prop}
Each of the relations $\le_{\forall}$ and $\le_{\exists}$ similarly
extends, both universally and existentially, to $\mathcal{P}(\mathcal{P}(\mathbb{Q}))$.
To be more specific, two collections $\mathcal{G},\mathcal{H}\subseteq\mathcal{P}(\mathbb{Q})$
satisfy 
\[
\mathcal{G}\le_{\forall\exists}\mathcal{H}
\]
 if 
\[
\forall G\in\mathcal{G},H\in\mathcal{H}:\quad G\le_{\exists}H.
\]
Similarly, 
\[
\mathcal{G}\le_{\exists\forall}\mathcal{H}
\]
if 
\[
\exists G\in\mathcal{G},H\in\mathcal{H}:\quad G\le_{\forall}H.
\]
The meaning of $\mathcal{G}<_{\forall\exists}\mathcal{H}$, $\mathcal{G}\ge_{\exists\forall}\mathcal{H},$
or other derived notions, is similarly defined. It is obvious that
one can further extend the ordering on $\mathbb{Q}$ to ever more
complicated nested collections of rationals, however we will only
require the level two extensions given above. Note that typically
these extended relations are not orderings, e.g., relations $\le_{\exists\cdots}$
starting with an existential extension are rarely transitive. 

Since real numbers are collections of subsets of $\mathbb{Q}$ we
thus obtain relations on the reals which we now investigate.

For the proof of the following result, which is pivotal for the rest
of the construction, recall that the intersection of filters is always
a filter but that the intersection of Cauchy filters need not be Cauchy.
\begin{lem}
\label{lem:antisymmetry}The relation $\le_{\forall\exists}$ on $\mathbb{R}$
is antisymmetric. \end{lem}
\begin{proof}
Suppose $a\le_{\forall\exists}b$ and $b\le_{\forall\exists}a$, and
consider the filter $a\cap b$. If it can be shown that $a\cap b$
is in fact a Cauchy filter, then as $a$ and $b$ are minimal Cauchy
filters it will follow that $a=a\cap b=b$, and with it the result.
Let then $\varepsilon>0$ be given. As $a$ and $b$ are Cauchy there
exist $p,q\in\mathbb{Q}$ such that $p_{\frac{\varepsilon}{2}}\in a$
and $q_{\frac{\varepsilon}{2}}\in b$. We may assume, without loss
of generality, that $p<q$. Since $b\le_{\forall\exists}a$ it follows
that $q_{\frac{\varepsilon}{2}}\le_{\exists}p_{\frac{\varepsilon}{2}}$
and thus (remember that $p<q$) that $p_{\frac{\varepsilon}{2}}\cap q_{\frac{\varepsilon}{2}}\ne\emptyset$,
so that we may find $r\in p_{\frac{\varepsilon}{2}}\cap q_{\frac{\varepsilon}{2}}$.
It is now elementary to verify that $r_{\varepsilon}\supseteq p_{\frac{\varepsilon}{2}}$,
and thus $r_{\varepsilon}\in a$. Similarly, $r_{\varepsilon}\in b$
and we may conclude, as was intended, that $a\cap b$ is Cauchy. 
\end{proof}
The following corollary is useful. 
\begin{thm}[Equality Criterion For Real Numbers]
\label{thm:CriterionForEquality}Two real numbers $a$ and $b$ are
equal if, and only if, $A\cap B\ne\emptyset$ for all $A\in a$ and
$B\in b$.\end{thm}
\begin{proof}
If $a=b$, then $A\cap B\ne\emptyset$ for all $A,B\in a$ simply
because $a$ is a proper filter. Conversely, if $A\cap B\ne\emptyset$
for all $A\in a$ and $B\in b$, then any $x\in A\cap B$ demonstrates
that $A\le_{\exists}B$ and $B\le_{\exists}A$. Consequently, $a\le_{\forall\exists}b$
and $b\le_{\forall\exists}a$, yielding $a=b$. 
\end{proof}
The stage is now set for introducing the total ordering on the reals. 
\begin{defn}
For real numbers $a,b\in\mathbb{R}$ we write $x<y$ if $x<_{\exists\forall}y$.\end{defn}
\begin{thm}
$(\mathbb{R},<)$ is a total ordering. \end{thm}
\begin{proof}
To show irreflexivity, assume that $a<a$. Then $A<_{\forall}A'$
for some $A,A'\in a$, but $A\cap A'\ne\emptyset$, clearly a contradiction.
Next, to show transitivity, assume $a<b<c$. Then $A<_{\forall}B'$
and $B''<_{\forall}C$ for some $A\in a$, $B',B''\in b$, and $C\in c$,
which are all necessarily non-empty. Taking $B=B'\cap B''$, which
is again non-empty, it follows that $A<_{\forall}B<_{\forall}C$ and
so the transitivity of $<_{\forall}$ on $\mathcal{P}(\mathbb{Q})\setminus\{\emptyset\}$
implies that $A<_{\forall}C$, and so $a<c$. To show asymmetry, suppose
that $a<b$ and $b<a$ both hold. Then $A<_{\forall}B$ and $B'<_{\forall}A'$
for some $A,A'\in a$ and $B,B'\in b$. Hence $A<_{\forall}B\cap B'<_{\forall}A'$
and, again, none of these sets is empty so we may conclude that $A<_{\forall}A'$.
But that implies that $a<a$, which was already seen to be impossible. 

The proof up to now only used the fact that the reals are modeled
by filters. To complete the proof we need to show that if $a\ne b$,
then either $a<b$ or $b<a$. It is here that the minimal Cauchy condition
plays a role (via \lemref{antisymmetry}). Indeed, if $a\nless b$
and $b\nless a$, then $a\ge_{\forall\exists}b$ and $b\ge_{\forall\exists}a$,
and thus $a=b$. 
\end{proof}

\subsection{The embedding of $\mathbb{Q}$ in $\mathbb{R}$}

Every $q\in\mathbb{Q}$ gives rise to two filters. One is the \emph{maximal
principal filter} $\langle q\rangle=\{S\subseteq\mathbb{Q}\mid q\in S\}$,
and the other is the \emph{minimal principal filter} $\iota(q)=\{S\subseteq\mathbb{Q}\mid q_{\varepsilon}\subseteq S,\varepsilon>0\}$.
Clearly $\iota(q)\subseteq\langle q\rangle$ and each filter is Cauchy. 
\begin{lem}
For all rational numbers $q$ the filter $\iota(q)$ is a real number.\end{lem}
\begin{proof}
One way to proceed is to show that $\iota(q)=\langle q\rangle_{\circ}$,
the roundification of the maximal filter $\langle q\rangle$. Alternatively,
we will show directly that $\iota(q)$ is a minimal Cauchy filter.
Suppose that $\mathcal{F}\subseteq\iota(q)$ is a Cauchy filter and
let $S\in\iota(q)$, namely $q_{\varepsilon}\subseteq S$ for some
$\varepsilon>0$. Our goal is to show that $S\in\mathcal{F}$. As
$\mathcal{F}$ is Cauchy, there exists $p\in\mathbb{Q}$ such that
$p_{\frac{\varepsilon}{2}}\in\mathcal{F}$, and, since $\mathcal{F}\subseteq\iota(q)$,
$q\in p_{\frac{\varepsilon}{2}}$. It now follows that $p_{\frac{\varepsilon}{2}}\subseteq q_{\varepsilon}\subseteq S$,
and, since $\mathcal{F}$ is a filter, $S\in\mathcal{F}$, as required.
\end{proof}
We thus obtain a function $\iota:\mathbb{Q}\to\mathbb{R}$. 
\begin{prop}
The function $\iota:\mathbb{Q}\to\mathbb{R}$ is an order embedding. \end{prop}
\begin{proof}
Assume that $p<q$ are given rational numbers and let $\varepsilon=\frac{q-p}{2}$.
Clearly $p_{\varepsilon}<_{\forall}q_{\varepsilon}$ and since $p_{\varepsilon}\in\iota(p)$
and $q_{\varepsilon}\in\iota(q)$, we conclude that $\iota(p)<_{\exists\forall}\iota(q)$,
namely that $\iota(p)<\iota(q)$. \end{proof}
\begin{lem}
For all $q\in\mathbb{Q}$ and $a\in\mathbb{R}$
\begin{enumerate}
\item $a<\iota(q)$ if, and only if, $a<_{\exists\forall}\{\{q\}\}$.
\item $\iota(q)<a$ if, and only if, $\{\{q\}\}<_{\exists\forall}a$.
\end{enumerate}
\end{lem}
\begin{proof}
In the $\implies$ direction, both arguments follow the exact same
pattern, i.e., if $a<\iota(q)$, then $A<_{\forall}S$ for some $A\in a$
and $S\in\iota(q)$, but then since $q\in S$ it follows that $A<_{\forall}\{q\}$,
and thus $a<_{\exists\forall}\{\{q\}\}$. The arguments in the other
direction are also similar to each other. Suppose $a<_{\exists\forall}\{\{q\}\}$
holds but $a<\iota(q)$ does not, i.e., either $a=\iota(q)$ or $a>\iota(q)$.
From $a<_{\exists\forall}\{\{q\}\}$ it follows that there exists
$A\in a$ with $A<_{\forall}\{q\}$. Suppose that $a>\iota(q)$ holds,
i.e., $a>_{\exists\forall}\{\{q\}\}$. There exist then $A'\in a$
such that $\{q\}<_{\forall}A'$. Considering $A\cap A'$, which is
non-empty, we have $\{q\}<_{\forall}A\cap A'<_{\forall}\{q\}$, which
is nonsense. Assume now that $a=\iota(q)$. But then $q\in A$ and
thus $A<_{\forall}\{q\}$ is impossible. We conclude that $a<\iota(q)$. 
\end{proof}
This result shows that we can quite safely abuse notation and identify
$\iota(q)$ with $q$. However, we resist the temptation of making
this identification quite yet and, for the sake of clarity, we opt
for the following definition. 
\begin{defn}
A real number $a$ is a \emph{rational real number }if there exists
$q\in\mathbb{Q}$ with $a=\iota(q)$. A real number that is not a
rational real number is called an \emph{irrational real number}.
\end{defn}
We now give an internal characterization of the rational real numbers.
For a filter $\mathcal{F}$, the intersection $C(\mathcal{F})=\bigcap_{F\in\mathcal{F}}F$
of all of its members is called the \emph{core }of $\mathcal{F}$.
The filter $\mathcal{F}$ is said to be \emph{free} if its core is
empty. 
\begin{lem}
\label{lem:RatIFFCisq}Let $a$ be a real number. 
\begin{itemize}
\item $a$ is a rational real number if, and only if, its core is a singleton
set, and in that case $a=\iota(q)$ if, and only if, $C(a)=\{q\}$. 
\item $a$ is an irrational real number if, and only if, it is a free filter.
\end{itemize}
\end{lem}
\begin{proof}
Firstly, we establish that the core of any Cauchy filter $\mathcal{F}$
is either empty or a singleton set. Indeed, suppose that $p,q\in C(\mathcal{F})$
and $p<q$,  and let $\varepsilon=\frac{q-p}{2}$. Since $\mathcal{F}$
is Cauchy, there exists an $x\in\mathbb{Q}$ with $x_{\varepsilon}\in\mathcal{F}$.
Since $p$ and $q$ are in the core of $\mathcal{F}$ it follows that
$p,q\in x_{\varepsilon}$, which by the choice of $\varepsilon$ is
impossible. 

Assume now that $a$ is rational, i.e., $a=\iota(q)$ for some $q\in\mathbb{Q}$.
It is obvious from the definition of $\iota(q)$ that $q\in C(\iota(q))$,
and thus necessarily $C(\iota(q))=\{q\}$. Conversely, if $C(a)=\{q\}$
for some $q\in\mathbb{Q}$, then $q\in A\cap B$ for all $A\in a$
and $B\in\iota(q)$. It follows from \thmref{CriterionForEquality}
that $a=\iota(q)$. The characterization of irrational real numbers
follows by contrapositives. 
\end{proof}
We conclude the treatment of the order on $\mathbb{R}$ and the place
of the rational real numbers within the reals by establishing density,
the archimedean property, and some useful technical results. 
\begin{prop}
\label{prop:InternalApproximations}The equivalence\label{prop:ComputingApprox}
\[
q_{\varepsilon}\in a\iff\iota(q-\varepsilon)<a<\iota(q+\varepsilon)
\]
holds for all real numbers $a$ and rational numbers $q$ and $\varepsilon>0$.\end{prop}
\begin{proof}
If $q_{\varepsilon}\in a$ then it follows from \lemref{RatIFFCisq}
and the obvious inequalities $\{q-\varepsilon\}<_{\forall}q_{\varepsilon}<_{\forall}\{q+\varepsilon\}$
that $\iota(q-\varepsilon)<a<\iota(q+\varepsilon)$. Conversely, suppose
that $\iota(q-\varepsilon)<a<\iota(q+\varepsilon)$. Then, again by
\lemref{RatIFFCisq}, there exist $A,A'\in a$ with $\{q-\varepsilon\}<_{\forall}A$
and $A'<_{\forall}\{q+\varepsilon\}$. Let $A_{0}=A\cap A'$ and then
$\{q-\varepsilon\}<_{\forall}A_{0}<_{\forall}\{q+\varepsilon\}$,
implying that $A_{0}\subseteq q_{\varepsilon}$. Since $A_{0}\in a$
it follows that $q_{\varepsilon}\in a$. \end{proof}
\begin{rem}
Notice the immediate equivalent formulation of this result, namely
that $q_{\varepsilon}\notin a$ if, and only if, either $a\le\iota(q-\varepsilon)$
or $a\ge\iota(q+\varepsilon)$. \end{rem}
\begin{cor}[Rational Approximations]
For all $a\in\mathbb{R}$ and $\varepsilon>0$ there exist a rational
number $q$ and $\varepsilon>0$ such that $\iota(q-\varepsilon)<a<\iota(q+\varepsilon)$.\end{cor}
\begin{proof}
Every real number $a\in\mathbb{R}$ is a Cauchy filter, and thus for
any $\varepsilon>0$ there exists $q\in\mathbb{Q}$ with $q_{\varepsilon}\in a$. \end{proof}
\begin{cor}
The rational real numbers are dense in $\mathbb{R}$. \end{cor}
\begin{proof}
Let $a<b$ be real numbers. There exist then $A\in a$ and $B\in b$
with $A<_{\forall}B$. Inside $A$ we may find an interval $p_{\varepsilon}\subseteq A$
with $p_{\varepsilon}\in a$. Similarly, there is an interval $q_{\delta}\subseteq B$
with $q_{\delta}\in b$. Clearly, $p_{\varepsilon}<_{\forall}q_{\delta}$,
which implies that $p+\varepsilon\le q-\delta$. We now have that
$a<\iota(p+\varepsilon)\le\iota(q-\delta)<b$, and so at least one
rational real number between $a$ and $b$ is found. \end{proof}
\begin{cor}
$\mathbb{R}$ is an archimedean order in the sense that for all $a\in\mathbb{R}$
there exists $n\in\mathbb{N}$ with $a<\iota(n)$. \end{cor}
\begin{proof}
If $a<\iota(q+\varepsilon)$, then any $n>q+\varepsilon$ will do. 
\end{proof}
Summarizing the results so far we see that the function $\iota:\mathbb{Q}\to\mathbb{R}$
is a dense order embedding. At this point it is a simple matter to
deduce the following useful result. 
\begin{prop}
For any real number $a$ the following hold \end{prop}
\begin{itemize}
\item There exists a rational number $M>0$ and $A\in a$ with $\{-M\}<_{\forall}A<_{\forall}\{M\}$.
\item $a>0$ if, and only if, for every $A\in a$ there exists $A_{+}\subseteq A$
with $A_{+}>_{\forall}\{0\}$ and $A_{+}\in a$.
\item $a<0$ if, and only if, for every $A\in a$ there exists $A_{-}\subseteq A$
with $A_{-}<_{\forall}\{0\}$ and $A_{-}\in a$.
\item $a=0$ if, and only if, $0\in A$ for all $A\in a$. \end{itemize}
\begin{proof}
For the first assertion, there exists $p\in\mathbb{Q}$ with $p_{1}\in a$,
so a choice for $M$ is clear. If $a>0$, then $a>_{\exists\forall}\{\{0\}\}$,
so there is some $A^{+}\in a$ with $A^{+}>_{\forall}\{0\}$. Given
any $A\in a$ one may then take $A_{+}=A\cap A^{+}$, which fulfills
the required condition. Conversely, the condition on $A_{+}$ immediately
implies that $a>_{\exists\forall}\{\{0\}\}$, and hence $a>0$. The
proof for negative numbers is similar, and the characterization of
$a=0$ is just a restatement of the characterization of all rational
real numbers in terms of cores. 
\end{proof}
The following technically sharper result on positive and negative
real numbers is also useful. 
\begin{prop}
If $a>0$ is a real number, then there exists a rational number $\eta>0$
and a rational number $\delta>0$ such that $p_{\delta'}\in a$ implies
$p_{\delta'}>_{\forall}\{\eta\}$, for all $p\in\mathbb{Q}$ and $0<\delta'\le\delta$.
A similar result holds for negative real numbers. \label{prop:PosImpliesDelta}\end{prop}
\begin{proof}
Let $\eta>0$ be a rational number with $a>\iota(3\eta)$. Then there
exists $A\in a$ such that $A>_{\forall}\{3\eta\}$. Let $\delta=\eta.$
Then, if $p_{\delta'}\in a$ with $0<\delta'\le\delta$, then $p_{\delta'}\cap A\ne\emptyset$,
implying $x>3\eta$ for some $x\in p_{\delta'}$, and as the length
of $p_{\delta'}$ is $2\delta'<2\eta$ it follows that $p_{\delta'}>_{\forall}\{\eta\}$,
as required.  
\end{proof}

\subsection{Arithmetic}

The addition operation of the rational numbers extends element-wise
to sets $A,B\subseteq\mathbb{Q}$, i.e., we define $A+B=\{\alpha+\beta\mid\alpha\in A,\beta\in B\}$.
Further, for arbitrary collections $\mathcal{F}$ and $\mathcal{G}$
of subsets of rational numbers, let $\mathcal{F}\oplus\mathcal{G}=\{A+B\mid A\in\mathcal{F},B\in\mathcal{G}\}$.
In particular, for real numbers $a,b\in\mathbb{R}$, we have the collection
$a\oplus b=\{A+B\mid A\in a,B\in b\}$. Similarly, multiplication
is also extended via $A\cdot B=AB=\{\alpha\beta\mid\alpha\in A,\beta\in B\}$
and $\mathcal{F}\odot\mathcal{G}=\{AB\mid A\in\mathcal{F},B\in\mathcal{G}\}$,
and in particular for real numbers then $a\odot b=\{AB\mid A\in a,B\in b\}$.
For $p\in\mathbb{Q}$ and $B\subseteq\mathbb{Q}$ we write $p+B$
as shorthand for $\{p\}+B$, and similarly $p\cdot B$ for $\{p\}\cdot B$. 
\begin{prop}
For all real numbers $a,b\in\mathbb{R}$, the collections $a\oplus b$
and $a\odot b$ are Cauchy filter bases. \end{prop}
\begin{proof}
Each collection is clearly non-empty. The condition for filter base
is verified by noting that for all $A,A'\in a$ and $B,B'\in b$ 
\begin{eqnarray*}
(A\cap A')+(B\cap B') & \subseteq & (A+B)\cap(A'+B')\\
(A\cap A')(B\cap B') & \subseteq & (AB)\cap(A'B')
\end{eqnarray*}
together with the fact that $A\cap A'\in a$ and $B\cap B'\in b$.
The Cauchy condition for $a\oplus b$ is immediate; for $\varepsilon>0$
there exist $p,q\in\mathbb{Q}$ with $p_{\frac{\varepsilon}{2}}\in a$
and $q_{\frac{\varepsilon}{2}}\in q$, and then $(p+q)_{\varepsilon}=p_{\frac{\varepsilon}{2}}+q_{\frac{\varepsilon}{2}}\in a\oplus b$.
As for $a\odot b$, fix $\varepsilon>0$, and a natural number $M$
together with $A\in a$ and $B\in b$ with $\{-M\}\le_{\forall}A,B\le_{\forall}\{M\}$,
and let $\delta=\frac{\varepsilon}{M+2\varepsilon}$. As $a$ and
$b$ are Cauchy, there exist $p,q\in\mathbb{Q}$ with $p_{\delta}\in a$
and $q_{\delta}\in b$. In particular, both intervals are in the range
$(-M-2\varepsilon,M+2\varepsilon)$ (since $p_{a}$ intersects $A$,
and $p_{b}$ intersects $B$) and thus the interval $p_{\delta}q_{\delta}$
has length bounded by $(M+2\varepsilon)\frac{\varepsilon}{(M+2\varepsilon)}=\varepsilon$,
as required. \end{proof}
\begin{defn}
Given real numbers $a,b\in\mathbb{R}$, their \emph{sum }is $a+b=\langle a\oplus b\rangle$
and their \emph{product }is $ab=\langle a\odot b\rangle$. In more
detail, a subset $C\subseteq\mathbb{Q}$ satisfies $C\in a+b$ (respectively
$C\in ab$) precisely when there exist $A\in a$ and $B\in b$ with
$C\supseteq A+B$ (respectively $C\supseteq AB$). 
\end{defn}
We will denote $0=\iota(0)$, unless doing so may cause confusion. 
\begin{prop}
The equalities $a\cdot\iota(0)=\iota(0)=\iota(0)\cdot a$ hold for
all $a\in\mathbb{R}$.\label{prop:prodWithZero}\end{prop}
\begin{proof}
Suppose that $B\in a\cdot\iota(0)$. Then $B\supseteq A\cdot0_{\varepsilon}$
for some $A\in a$ and $\varepsilon>0$. Further, $A\ne\emptyset$
and for any $\alpha\in A$ we have $B\supseteq\alpha\cdot0_{\varepsilon}=0_{\alpha\varepsilon}$,
and thus $B\in\iota(0)$, showing that $a\cdot\iota(0)\subseteq\iota(0)$.
In the other direction, if $B\in\iota(0)$, then $B\supseteq0_{\varepsilon}$,
with $\varepsilon>0$. Let now $A\in a$ and $M>0$ with $\{-M\}<_{\forall}A<_{\forall}\{M\}$,
and consider $\delta=\frac{\varepsilon}{M}$. It follows easily that
$0_{\varepsilon}\supseteq A\cdot0_{\delta}$, and thus $0_{\varepsilon}\in a\cdot\iota(0)$,
as needed. The proof that $\iota(0)=\iota(0)\cdot a$ is similar. \end{proof}
\begin{lem}
\label{lem:sumOfReals}The sum and product of any two real numbers
$a$ and $b$ are real numbers. \end{lem}
\begin{proof}
As $a\oplus b$ and $a\odot b$ are Cauchy filter bases, it follows
that $a+b$ and $ab$ are Cauchy filters so it only remains to be
shown that each of these filters is also round. We start with $a+b$.
Let $C\in a+b$, i.e., $C\supseteq A+B$ for some $A\in a$ and $B\in b$.
Our goal is to find $\delta>0$ such that $p_{\delta}\in a+b$ implies
$p_{\delta}\subseteq C$, for all $p\in\mathbb{Q}$. Since $a$ is
round there exists $\varepsilon'>0$ for which $x_{\varepsilon'}\in a$
implies $x_{\varepsilon'}\subseteq A$, for all $x\in\mathbb{Q}$.
Similarly, there exists $\varepsilon''>0$ such that $x_{\varepsilon''}\in b$
implies $x_{\varepsilon''}\subseteq B$, for all $x\in\mathbb{Q}$.
Let $\delta=\min\{\varepsilon',\varepsilon''\}$. Suppose now that
$p_{\delta}\in a+b$ holds for some $p\in\mathbb{Q}$, namely $p_{\delta}\supseteq A'+B'$
for some $A'\in a$ and $B'\in b$. We will conclude the proof by
showing that $p_{\delta}\subseteq C$, which will be achieved by showing
that $p_{\delta}\subseteq A+B$. It is easily seen that $p_{\delta}\subseteq y_{2\delta}$
for any $y\in p_{\delta}$, and so it suffices to find $y\in A'+B'$
with $y_{2\delta}\subseteq A+B$. For all $\alpha'\in A'$ and $\beta'\in B'$
we have that $(\alpha'+\beta')_{2\delta}=\alpha'_{\delta}+\beta'_{\delta}\subseteq\alpha'_{\varepsilon_{1}}+\beta'_{\varepsilon_{2}}$,
so the problem is now reduced to finding $\alpha'\in A'$ and $\beta'\in B'$
with $\alpha'_{\varepsilon_{1}}\subseteq A$ and $\beta'_{\varepsilon_{2}}\subseteq B$.
The existence of such elements is verified as follows. Let $\alpha'\in\mathbb{Q}$
and $r>0$ with $\alpha'_{r}\in a$ and $\alpha'_{r}\subseteq A'$.
Since $A'+B'\subseteq p_{\delta}$ it is seen that $\alpha'_{r}$
translates (by any element in $B'$) into $p_{\delta}$, and thus
$r\le\delta$. We now know that $\alpha'_{r}\subseteq\alpha'_{\delta}\subseteq\alpha'_{\varepsilon'}$,
and thus that $\alpha'_{\varepsilon'}\in a$, and from the choice
of $\varepsilon'$ we conclude that $\alpha'_{\varepsilon'}\subseteq A$.
The existence of $\beta'\in B'$ with $\beta'_{\varepsilon''}\subseteq B$
follows similarly, and the proof is now complete. The proof for the
product follows along the same lines, utilizing the fact that for
every real number $c$ there exists a natural number $M$ and $C\in c$
with $\{-M\}<_{\forall}A<_{\forall}\{M\}$ in order to obtain correct
bounds. We omit the details. \end{proof}
\begin{prop}
If $a=\iota(p)$ is a rational real number and $b$ is any real number,
then $\iota(p)+b=\{p+B\mid B\in b\}$ and $\iota(p)\cdot b=\{p\cdot B\mid B\in b\}$.\label{prop:AddingMultRationals}\end{prop}
\begin{proof}
The fact that the collection $\{p+B\mid B\in b\}$ is a real number,
i.e., that it is a filter, that it is Cauchy, and that it is round,
are immediate. To show that $\{p+B\mid B\in b\}=\iota(p)+b$ we apply
\thmref{CriterionForEquality}. Given an arbitrary $p+B$ and an arbitrary
$x\in\iota(p)+b$, namely $x\supseteq A+B'$ with $A\in\iota(p)$
and $B'\in b$, we need to show that $(p+B)\cap(A+B')\ne\emptyset$.
But since $p\in A$ for all $A\in\iota(p)$ and since $B\cap B'\ne\emptyset$
that claim is obvious. For the multiplicative part of the claim, note
that the case $p=0$ is \propref{prodWithZero}. For $p\ne0$ it is
straightforward that $\{p\cdot B\mid B\in b\}$ is a real number and
the rest of the proof is virtually the same as the additive claim. \end{proof}
\begin{thm}
With addition and multiplication of real numbers, $\mathbb{R}$ is
a commutative ring with unity. \end{thm}
\begin{proof}
Let $a,b,c\in\mathbb{R}$ be given. We note first that 
\begin{itemize}
\item $\langle a\oplus(b+c)\rangle=\langle a\oplus(b\oplus c)\rangle$ 
\item $\langle(a+b)\oplus c\rangle=\langle(a\oplus b)\oplus c\rangle$
\item $\langle a\odot(bc)\rangle=\langle a\odot(b\odot c)\rangle$
\item $\langle(ab)\odot c\rangle=\langle(a\odot b)\odot c\rangle$
\item $\langle a\odot(b+c)\rangle=\langle a\odot(b\oplus c)\rangle$.
\end{itemize}
Indeed, all of these equalities are established by essentially the
same argument, so we only verify the first one. Since $b\oplus c\subseteq b+c$,
one of the inclusions is trivial. For the other inclusion, a typical
element $X$ in $\langle a\oplus(b+c)\rangle$ is a subset of $\mathbb{Q}$
with $X\supseteq A+Y$ for some $A\in a$ and $Y\in b+c$. But then
$Y$ itself contains a set of the form $B+C$ for $B\in b$ and $C\in c$,
and thus $X\supseteq A+(B+C)$, and is thus in $a\oplus(b\oplus c)$.
Associativity of $+$ now follows at once since 
\[
a+(b+c)=\langle a\oplus(b+c)\rangle=\langle a\oplus(b\oplus c)\rangle=\langle(a\oplus b)\oplus c\rangle=\langle(a+b)\oplus c\rangle=(a+b)+c
\]
as $(a\oplus b)\oplus c=a\oplus(b\oplus c)$ is immediate since obviously
$(A+B)+C=A+(B+C)$ holds for all subsets of $\mathbb{Q}$. Repetition
of this argument shows that multiplication is associative, that both
addition and multiplication are commutative, and the distributivity
law. \propref{AddingMultRationals} implies at once the neutrality
of $0=\iota(0)$ and that $1=\iota(1)$ is a multiplicative identity
element. Finally, for the existence of additive inverses, consider
a real number $a$ and let $b=\{-A\mid A\in a\}$, where $-A=\{-\alpha\mid\alpha\in A\}$.
Obviously, $b$ is a real number and we now show that $a+b=0$. By
\lemref{RatIFFCisq} it suffices to show that $0\in C$ for all $C\in a+b$.
Indeed, for such a $C$ there exists $A,A'\in a$ with $C\supseteq A+(-A')\supseteq(A\cap A')+(-(A\cap A'))$,
a set which certainly contains $0$ since $A\cap A'\ne\emptyset$. \end{proof}
\begin{thm}
$\mathbb{R}$ With addition and multiplication is a field. \end{thm}
\begin{proof}
Fix a real number $a>0$, for which we shall present an inverse $a^{-1}$.
For an arbitrary $A\subseteq\mathbb{Q}$ let $\frac{1}{A}=\{\frac{1}{\alpha}\mid\alpha\in A,\alpha\ne0\}$.
Noting that $\frac{1}{A}\cap\frac{1}{A'}=\frac{1}{A\cap A'}$, it
follows at once that $\mathcal{B}=\{\frac{1}{A}\mid A\in a\}$ is
a filter base. We proceed to show that it is Cauchy, so let us fix
an $\varepsilon>0$. By \propref{PosImpliesDelta} there exists $\eta>0$
and $\delta>0$ such that $p_{\delta'}\in a$ implies $p_{\delta'}>_{\forall}\{\eta\}$,
for all $p\in\mathbb{Q}$ and $0<\delta'\le\delta$. It then follows
that $\frac{1}{p_{\delta'}}$ is an interval whose length is $\frac{2\delta'}{(p-\delta')(p+\delta')}\le\frac{2\delta'}{\eta^{2}}<\varepsilon$,
for a sufficiently small $\delta'>0$. The existence of some $p_{\delta'}\in a$
is guaranteed since $a$ is Cauchy. 

We may now define $a^{-1}=\langle\frac{1}{A}\mid A\in a\rangle$,
the generated filter (with slight abuse of notation), which is thus
Cauchy. It is a bit tedious to show directly that $\mathcal{B}$ is
also a round filter base. To avoid these details, and since we are
only interested in the existence of a multiplicative inverse, let
us consider $a^{-1}=\langle\mathcal{B}\rangle_{\circ}$, the roundification
of the generated filter, which is thus both Cauchy and round, and
hence a real number. To show that $a\cdot a^{-1}=1=\iota(1)$ it suffices
to compute the core and appeal to \lemref{RatIFFCisq}. Indeed, since
$\langle\mathcal{B}\rangle_{\circ}\subseteq\langle\mathcal{B}\rangle$,
given $C\in a\cdot a^{-1}$ there exist $A,A'\in a$ such that $C\supseteq A\cdot\frac{1}{A'}$
which contains $(A\cap A')\cdot\frac{1}{A\cap A'}$, which itself
clearly contains $1$ since $A\cap A'\ne\emptyset$, and in fact contains
at least two elements, one of which is not $0$. 

The proof for $a<0$ is similar. \end{proof}
\begin{thm}
$\mathbb{R}$ with addition and multiplication is an ordered field.\end{thm}
\begin{proof}
Let $a,b,c\in\mathbb{R}$ with $a\le b$. We have to show that $a+c\le b+c$
and, if $c>0$, that $ac\le bc$. Indeed, given arbitrary $D\in a+c$
and $D'\in b+c$ there exist $A\in a$, $B\in b$, and $C,C'\in c$
with $D\supseteq A+C$ and $D'\supseteq B+C'$. Further, since $a\le_{\forall\exists}b$,
it holds that $A\le_{\exists}B$, so that $\alpha\le\beta$ for some
$\alpha\in A$ and $\beta\in B$. Since $C\cap C'\ne\emptyset$ let
$\gamma\in C\cap C'$. Then $\alpha+\gamma\le\beta+\gamma$, showing
that $A+C\le_{\exists}B+C'$, and thus that $D\le_{\exists}D'$. As
$D$ and $D'$ were arbitrary we showed that $a+c\le_{\forall\exists}b+c$,
as required. 

The argument for showing that $ac\le bc$ under the further condition
$c>0$ is similar. Firstly, since $c>0$ there exists $C_{+}\in c$
with $C_{+}>_{\forall}\{0\}$. Now, given arbitrary $D\in ac$ and
$D'\in bc$ there exist $A\in a$, $B\in b$, and $C,C'\in c$ with
$D\supseteq AC$ and $D'\supseteq BC'$. As above, we have $\alpha\le\beta$
for some $\alpha\in A$ and $\beta\in B$. As $C\cap C'\cap C_{+}$
is non-empty, let $\gamma\in C\cap C'\cap C_{+}$. Then $\alpha\gamma\le\beta\gamma$,
showing that $AC\le_{\exists}BC'$, and thus that $D\le_{\exists}D'$,
which were arbitrary and thus $ac\le_{\forall\exists}bc$. The proof
is complete. \end{proof}
\begin{cor}
The canonical embedding $\iota:\mathbb{Q}\to\mathbb{R}$ is a field
homomorphism. \end{cor}
\begin{proof}
We have to show that $\iota(p+q)=\iota(p)+\iota(q)$ and that $\iota(pq)=\iota(p)\iota(q)$,
for all $p,q\in\mathbb{Q}$. Indeed, since $C(\iota(p))=\{p\}$ and
$C(\iota(q))=\{q\}$ it follows at once that $p+q\in C(\iota(p)+\iota(q))$
and that $pq\in C(\iota(p)\iota(q))$. The claim now follows by \lemref{RatIFFCisq}. 
\end{proof}

\subsection{Completeness}

We now establish the completeness property of the reals. Let us fix
a non-empty set $\mathcal{A}$ of real numbers and assume that it
is bounded above by some real number $c$. Consider the collection
$\mathbb{A}=\{p_{\varepsilon}\mid p_{\varepsilon}\in a_{0},a_{0}\in\mathcal{A}\}$,
which represents an attempt to collate all of $\mathcal{A}$ into
a single real number. However, this collection fails to be a filter.
We thus refine it by considering the collection $\mathcal{B}=\{p_{\varepsilon}\in\mathbb{A}\mid a<\iota(p+\varepsilon),\forall a\in\mathcal{A}\}$,
which is non-empty since $\mathcal{A}$ is non-empty and bounded above.
Intuitively, the condition sifts away those elements in $\mathbb{A}$
which lie too far below in $\mathcal{A}$. The proof of completeness
proceeds in two steps:
\begin{enumerate}
\item Establish that $\mathcal{B}$ is a Cauchy filter. 
\item Prove that $b=\langle\mathcal{B}\rangle_{\circ}$ is the least upper
bound of $\mathcal{A}$.
\end{enumerate}
Let us first tend to the second task as it is quite straightforward.
Working under the supposition that $\mathcal{B}$ is a Cauchy filter
base, it follows that $\langle\mathcal{B}\rangle$, the generated
filter, is a Cauchy filter. The filter $b=\langle\mathcal{B}\rangle_{\circ}$
is thus Cauchy and round, namely a real number. Recalling that the
roundification of a filter $\mathcal{F}$ always yields a subfilter
of $\mathcal{F}$ we have that $b\subseteq\langle\mathcal{B}\rangle$.
Consequently, for every $B\in b$ there exists $q\in\mathbb{Q}$ and
$\varepsilon>0$ with $B\supseteq q_{\varepsilon}\in\mathcal{B}$.
We call any such $q_{\varepsilon}$ a \emph{witnessing interval }for
$B$. 

Given an arbitrary $a_{0}\in\mathcal{A}$ we need to show that $a_{0}\le_{\forall\exists}b$,
so let us fix $A\in a_{0}$ and $B\in b$, and we need to establish
that $A\le_{\exists}B$. Let $q_{\varepsilon}$ be a witnessing interval
for $B$. Since $q_{\varepsilon}\in\mathcal{B}$ it follows that $a_{0}<\iota(q+\varepsilon)$.
It now follows that there exists $A'\in a_{0}$ with $A'<_{\forall}\{q+\varepsilon\}$,
and let $A_{0}=A\cap A'$, which is in $a_{0}$ and thus non-empty.
It suffices to show now that $A_{0}\le_{\exists}q_{\varepsilon}$.
But for any $\alpha\in A_{0}$ it follows that $\alpha<q+\varepsilon$,
and thus $\alpha\le y$ for some $y\in q_{\varepsilon}$, as required.
To complete the argument that $b$ is the least upper bound, suppose
that $c$ is any upper bound of $\mathcal{A}$ and assume that $c<b$.
There exist then $C\in c$ and $B\in b$ with $C<_{\forall}B$. Taking
a witnessing interval $q_{\varepsilon}$ for $B$ we have that $C<_{\forall}q_{\varepsilon}$,
and $q_{\varepsilon}\in a_{0}$ for some $a_{0}\in\mathcal{A}$. However,
$c$ is an upper bound of $\mathcal{A}$ and thus $c\ge a_{0}$, namely
$c\ge_{\forall\exists}a_{0}$. It follows that $C\ge_{\exists}q_{\varepsilon}$,
clearly contradicting $C<_{\forall}q_{\varepsilon}$. 

It now remains to show that $\mathcal{B}$ is a Cauchy filter base.
Firstly, $\mathcal{B}$ is a filter base as follows. Suppose $p_{\varepsilon},q_{\delta}\in\mathcal{B}$,
with $p_{\varepsilon}\in a'$ and $q_{\delta}\in a''$, and, without
loss of generality, $a'\le a''$. Now, the intersection $p_{\varepsilon}\cap q_{\delta}$,
if not empty, is an interval $s_{\eta}$ whose upper bound $s+\eta$
is either $p+\varepsilon$ or $q+\delta$, and thus the condition
$a<\iota(s+\eta)$ holds for all $a\in\mathcal{A}$ automatically
since it holds for both $p+\varepsilon$ and $q+\delta$. Hence, to
conclude that $s_{\eta}\in\mathcal{B}$ it suffices to show that $p_{\varepsilon}\in a''$,
since then $p_{\varepsilon}\cap q_{\delta}\in a''$ too, and is thus
non-empty. With the aid of \propref{ComputingApprox} we have $\iota(p-\varepsilon)<a'$,
and we wish to show that $\iota(p-\varepsilon)<a''<\iota(p+\varepsilon)$.
But $a''<\iota(p+\varepsilon)$ is immediate from $p_{\varepsilon}\in\mathcal{B}$,
while $\iota(p-\varepsilon)<a'<a''$ follows from the preceding inequalities. 

To show that $\mathcal{B}$ is Cauchy let $\varepsilon>0$ be given
and let $\delta=\frac{\varepsilon}{2}$. For each $a\in\mathcal{A}$
we may find $p(a)\in\mathbb{Q}$ such that $p(a)_{\delta}\in a$,
and in particular $\iota(p(a)+\delta)>a$. It suffices to exhibit
a single $a_{0}\in\mathcal{A}$ for which $\iota(p(a_{0})+\varepsilon)>a$,
for all $a\in\mathcal{A}$, since then (as $p(a_{0})_{\varepsilon}\supseteq p(a_{0})_{\delta}$)
$p(a_{0})_{\varepsilon}\in a_{0}$ and thus $p(a_{0})_{\varepsilon}\in\mathcal{B}$,
as required for the Cauchy condition. Suppose to the contrary that
for all $a\in\mathcal{A}$ there exists $a'\in\mathcal{A}$ with $a'\ge\iota(p(a)+\varepsilon)$.
In particular, $a'\ge\iota(p(a)+2\delta)>a+\iota(\delta)$. Starting
with an arbitrary $a_{0}\in\mathcal{A}$, an inductive argument then
shows that for arbitrary $n\in\mathbb{N}$ an element $a\in\mathcal{A}$
exists with $a>a_{0}+\iota(n\delta)$. But as $b$ is an upper bound
of $\mathcal{A}$ it follows that $b\ge a_{0}+\iota(n\delta)$, for
all $n\ge1$, contradicting the fact that $\mathbb{R}$ is archimedean.

\section{Consequences\label{sec:Consequences}}

In this final section we inspect two aspects of the real numbers through
the lens of the above formalism; the uncountability of the reals and
the limit definition. The criterion for equality of real numbers plays
a crucial role. 

Proofs of the uncountability of the reals are in (relative) abundance.
Typically, such a proof relies either on a particular representation
of real numbers (e.g., Cantor's famous diagonalization proof on the
digits of expansions of the real numbers) or on some (more or less
immediate) properties of the real number system (e.g., the reals are
uncountable since $\mathbb{R}$ is a non-empty complete metric space,
applying Baire's category theorem). 

The construction of the reals above lends itself straightforwardly
to yet another proof which uses only basic properties of filters and
the criterion for equality (\thmref{CriterionForEquality}). We first
observe a trivial auxiliary result. For rational intervals $I$ and
$J$ let $|I|$ denote the length of the interval, and we write $J\subseteq_{d}I$
if $J$ is \emph{deeply contained} in $I$, meaning, assuming $I=(p,q)$,
that there exists $\varepsilon>0$ such that $J\subseteq(p+\varepsilon,q-\varepsilon)$.
Then given a real number $a$ and a rational interval $I$, there
exists a rational interval $J\subseteq_{d}I$ with $|J|=|I|/5$, and
$A\in a$ with $J\cap A=\emptyset$. Indeed, since $a$ is Cauchy,
there exists some interval $A\in a$ with $|A|\le|I|/5$. Subdividing
$I$ into five equal intervals, any of the middle three is deeply
contained in $I$, and $A$ can not have non-empty intersection with
all three. 
\begin{thm}
\label{thm:uncount}The set $\mathbb{R}$ is uncountable. \end{thm}
\begin{proof}
Given a sequence $(a_{n})_{n\ge1}$ of real numbers it suffices to
construct a real number not in the sequence. Choose an arbitrary rational
interval $I_{0}$. Suppose that a sequence $I_{n}\subseteq_{d}I_{n-1}\subseteq_{d}\cdots\subseteq_{d}I_{0}$
was constructed together with $A_{1},\cdots,A_{n}$, where $A_{k}\in a_{k}$,
$1\le k\le n$, and with the properties that $|I_{k+1}|=|I_{k}|/5$
for all $0\le k<n$, and such that $I_{n}\cap A_{n}=\emptyset$ for
all $1\le k\le n$. Considering the real number $a_{n+1}$ and the
interval $I_{n}$ we may find a rational interval $I_{n+1}\subseteq_{d}I_{n}$
with $|I_{n+1}|=|I_{n}|/5$ and $A_{n+1}\in a_{n+1}$ so that $I_{n+1}\cap A_{n}=\emptyset$.
Continuing in this fashion, the sequence $\{I_{n}\}_{n\ge0}$ is a
Cauchy sequence of rational intervals which forms a round Cauchy filter
base, thus generating a real number $a$. Since $I_{n}\in a$ and
$I_{n}\cap A_{n}=\emptyset$, we conclude that $a\ne a_{n}$, for
all $n\ge1$. 
\end{proof}
A pleasant consequence of the filters formalism is the following reformulation
of the definition of convergence of a sequence of real numbers. The
lack of an explicit appearance of $\varepsilon$ should be noted. 
\begin{lem}
A sequence $(a_{n})_{n\ge1}$ of real numbers converges to a real
number $b$ if, and only if, for all $B\in b$ there exists $n_{0}\in\mathbb{N}$
such that $B\in a_{n}$, for all $n\ge n_{0}$.\end{lem}
\begin{proof}
Assume the condition holds. Given $\varepsilon>0$ consider $q\in\mathbb{Q}$
with $q_{\varepsilon/2}\in b$, which is equivalent to $\iota(q-\varepsilon/2)<b<\iota(q+\varepsilon/2)$.
Let $n_{0}\in\mathbb{N}$ be given by the assumed condition. Then,
for each $n\ge n_{0}$, $q_{\varepsilon/2}\in a_{n}$ and so $\iota(q-\varepsilon/2)<a_{n}<\iota(q+\varepsilon/2)$.
Together with the estimate for $b$ we obtain $b-\iota(\varepsilon)<a_{n}<b+\iota(\varepsilon)$,
the familiar convergence formulation. 

Conversely, suppose that $a_{n}\to b$ in the usual sense, and let
$B\in b$. By roundness and the Cauchy condition we may then find
$q\in\mathbb{Q}$ and $\varepsilon>0$ with $q_{\varepsilon/2}\in b$
and $q_{\varepsilon}\subseteq B$. In particular, $\iota(q-\varepsilon/2)<b<\iota(q+\varepsilon/2)$.
Let $n_{0}\in\mathbb{N}$ be such that $b-\iota(\varepsilon/2)<a_{n}<b+\iota(\varepsilon/2)$,
for all $n\ge n_{0}$. Then $q-\iota(\varepsilon)<a_{n}<q+\iota(\varepsilon)$,
and thus $q_{\varepsilon}\in a_{n}$, for all $n\ge n_{0}$. Since
$q_{\varepsilon}\subseteq B$ it follows that $B\in a_{n}$, for all
$n\ge n_{0}$, as required. 
\end{proof}
Note that one may now establish the uniqueness of limits as follows.
If $b,c$ are both limits of the sequence $(a_{n})_{n\ge1}$, then,
for all $B\in b$ and $C\in c$, there exists $n_{0}\in\mathbb{N}$
with, in particular, $B,C\in a_{n}$. Since $a_{n}$ is a proper filter,
it follows that $B\cap C\ne\emptyset$, and thus $b=c$ by the equality
criterion for real numbers. 

Much of the fundamentals of analysis can effectively be developed
along similar lines. For instance we mention the following. Consider
a sequence $(a_{n})_{n\ge1}$ of real numbers. Let $\hat{a}$ be the
collection of all subsets $\hat{A}\subseteq\mathbb{Q}$ for which
there exists $n_{0}\in\mathbb{N}$ with $\hat{A}\in a_{n}$, for all
$n\ge n_{0}$. It is a trivial matter to verify that $\hat{a}$ is
a proper filter. We leave the details of the following claim for the
amusement of the reader. Let us denote by $\mathcal{F}_{\infty}$
the filter generated by $\{(q,\infty)\subseteq\mathbb{Q}\mid q\in\mathbb{Q}\}$,
with $\mathcal{F}_{-\infty}$ defined similarly. 
\begin{thm}
Let $(a_{n})_{n\ge1}$ and $\hat{a}$ be as above. Then
\begin{enumerate}
\item $(a_{n})_{n\ge1}$ converges $\iff$ $\hat{a}$ is a Cauchy filter
$\iff(a_{n})$ is a Cauchy sequence.
\item if $(a_{n})_{n\ge1}$ converges, then it converges to $(\hat{a})_{\circ}$.
\item $c\in\mathbb{R}$ is a partial limit of $(a_{n})_{n\ge1}$ $\iff$
$c\supseteq\hat{a}$. 
\item $\infty$ is a partial limit of $(a_{n})_{n\ge1}$ $\iff$ $\mathcal{F}_{\infty}\supseteq\hat{a}$
(and similarly for $-\infty$). 
\end{enumerate}
\end{thm}
Our presentation of the construction of the reals has come to an end.
We hope the reader enjoyed its geometric flavour and the inherent
strong appeal to rational approximations, and hopefully she found
the proofs and progression elegant. We conclude with the pedagogical
remark that the techniques one learns in the course of the construction
(namely getting acquainted with filters) are useful in topology and
in analysis and are not ad-hoc tools just for this construction. Moreover,
the rather clean convergence criterion obtained in this section may
indicate that this particular construction of the real numbers can
serve as a firm bridge toward the study of elementary analysis rather
than being a swamp of convoluted details one never wants to see again
(or ever).

\subsection*{Ackknowledgements}

I wish to thank the referee for suggesting several of the topics that
went into the lengthy historical overview. With much appreciation
I thank Rahel Berman for comments and discussions. 

\bibliographystyle{plain}
\bibliography{generalReferences}

\end{document}